\documentclass[11pt]{article}

\usepackage{amsfonts,amsthm,amsmath,amssymb,euscript,mathrsfs}
\usepackage{etoolbox} 
\usepackage[shortlabels]{enumitem} 
\usepackage[authoryear]{natbib}
\usepackage{mathdots, mathtools}
\usepackage{color}
\RequirePackage[colorlinks,citecolor=blue,urlcolor=blue,breaklinks=true]{hyperref} 
\RequirePackage{hypernat} 

\AtBeginEnvironment{pmatrix}{\everymath{\displaystyle}}

\def\bbB{\mathbb{B}}

\def\bbE{\mathbb{E}}
\def\bbF{\mathbb{F}}

\def\bbI{\mathbb{I}}

\def\bbN{\mathbb{N}}

\def\bbP{\mathbb{P}}
\def\bbQ{\mathbb{Q}}
\def\bbR{\mathbb{R}}

\def\bbV{\mathbb{V}}

\def\cA{{\mathcal A}}

\def\cF{{\mathcal F}}

\def\cL{{\mathcal L}}
\def\cM{{\mathcal M}}
\def\cN{{\mathcal N}}

\def\cP{{\mathcal P}}

\def\sB{\mathscr{B}}

\def\sE{\mathscr{E}}

\def\sL{\mathscr{L}}

\def\sS{\mathscr{S}}

\def\gd{\delta}
\def\ge{\epsilon}

\def\gk{\kappa}

\def\gm{\mu}

\def\gq{\theta}

\def\gs{\sigma}
\def\gt{\tau}

\def\gD{\Delta}

\def\gO{\Omega}

\def\gQ{\Theta}

\def\ds{\ensuremath{\displaystyle}}
\newcommand{\dd}{\mathop{}\!\mathrm{d}}

\newcommand{\tup}{\alpha} 
\newcommand{\tdown}{\beta} 
\newcommand{\s}{\gamma} 
\newcommand{\uu}{b} 
\newcommand{\bP}[1]{\mathbb{P}_{#1}} 
\newcommand{\bPT}[2]{\mathbb{P}_{#1}^{#2}}

\newcommand{\sref}[1]{Section~\ref{#1}}
\newcommand{\tref}[1]{Table~\ref{#1}}

\newcommand{\cref}[1]{Chapter~\ref{#1}}
\newcommand{\dref}[1]{Definition~\ref{#1}}
\newcommand{\lemmaref}[1]{Lemma~\ref{#1}}
\newcommand{\propref}[1]{Proposition~\ref{#1}}
\newcommand{\thmref}[1]{Theorem~\ref{#1}}

\newtheorem{theorem}{Theorem}[section]
\newtheorem{corollary}[theorem]{Corollary} 
\newtheorem{lemma}[theorem]{Lemma}
\newtheorem{remark}[theorem]{Remark}
\newtheorem{proposition}[theorem]{Proposition}
\newtheorem{definition}[theorem]{Definition}

\begin{document}

\title{The mutual arrangement of Wright--Fisher diffusion path measures and its impact on parameter estimation}
\author{Paul A.\ Jenkins\footnote{{\small Dept of Statistics \& Dept of Computer Science, University of Warwick, United Kingdom. \href{mailto:p.jenkins@warwick.ac.uk}{p.jenkins@warwick.ac.uk}.}}}
\date{}
\maketitle
\begin{abstract}
The Wright--Fisher diffusion is a fundamentally important model of evolution encompassing genetic drift, mutation, and natural selection. Suppose you want to infer the parameters associated with these processes from an observed sample path. Then to write down the likelihood one first needs to know the mutual arrangement of two path measures under different parametrizations; that is, whether they are absolutely continuous, equivalent, singular, and so on. In this paper we give a complete answer to this question by finding the separating times for the diffusion---the stopping time before which one measure is absolutely continuous with respect to the other and after which the pair is mutually singular. In one dimension this extends a classical result of Dawson on the local equivalence between neutral and non-neutral Wright--Fisher diffusion measures. Along the way we also develop new zero-one type laws for the diffusion on its approach to, and emergence from, the boundary. As an application we derive an explicit expression for the joint maximum likelihood estimator of the mutation and selection parameters and show that its convergence properties are closely related to the separating time. 
\end{abstract}

\section{Introduction}
Consider the solution to a scalar, time-homogeneous, stochastic differential equation (SDE) on an interval $I = (l,r) \subseteq \bbR$:
\begin{align}
	\label{eq:SDE}
\dd X_t &= \gm_\theta(X_t)\dd t + \gs(X_t) \dd W_t, & X_0 &= x_0 \in {I},
\end{align}
with drift and diffusion coefficients $\gm_\theta:I\to\bbR$ and $\gs:I\to[0,\infty)$, and a parameter $\theta\in\Theta$. Suppose we are interested in statistical inference of a model $\{\bPT{\gq}{T}:\:\gq\in\Theta\}$ given observations from a sample path $X^T = (X_t:\: t\in[0,T])$, where $\bPT{\gq}{T}$ denotes the probability measure on sample paths associated with \eqref{eq:SDE} up to time $T$. In the ideal situation in which $X^T$ itself is observed, the likelihood for $\gq$ is given by the Radon--Nikodym derivative $\dd\bPT{\gq}{T}/\dd\bbQ^{T}$ for some dominating measure $\bbQ^T$, typically taken to be a fixed hypothesis $\bPT{\gq_0}{T}$, with respect to which $\bPT{\gq}{T}$ is absolutely continuous. Even when $X^T$ is not observed directly and we have only discrete and/or noisy observations from it, it is a common strategy to treat the whole path as an unobserved latent variable, to then augment the state space with this path, and finally to integrate over it via Monte Carlo \citep[e.g.][]{pap:rob:2012, pap:etal:2013, van:sch:2018, jen:etal:2023}; again we must first know whether $\bPT{\gq}{T}$ is absolutely continuous with respect to $\bPT{\gq_0}{T}$. That being the case a maximum likelihood estimator is given by
\[
\widehat{\theta} \coloneq \arg\max_{\theta\in\Theta}\frac{\dd\bPT{\gq}{T}}{\dd\bPT{\gq_0}{T}}.
\]
The goal for the present work is to study in detail when absolute continuity with respect to a dominating measure does and does not hold in the context of the \emph{Wright--Fisher diffusion}, and to find new zero-one type laws which can testify to the mutual arrangement (equivalence, mutual singularity, absolute continuity, etc.)\ of the two measures. From this we can derive the estimator $\widehat{\theta}$ and study its properties. As we discover below, absolute continuity or lack thereof also plays a key role in the estimator's convergence properties.

The Wright--Fisher diffusion is a fundamental model for the evolution of the frequency $X_t$ at time $t$ of a genetic variant, or \emph{allele}, in a large, randomly mating population (for an introduction to this diffusion in the context of population genetics see for example \citet{eth:2011}). There is a great deal of interest in inferring evolutionary parameters governing this diffusion from samples of DNA sequence data. With the advent of  technologies to sequence ancient DNA, such data is becoming increasingly available in time series form \citep{deh:etal:2020}. Though our focus is in its role as a model for evolution, we note that the Wright--Fisher diffusion has applications in other fields including biophysics \citep{dan:etal:2012:PCS}, neuroscience \citep{don:etal:2018}, mathematical finance \citep{del:shi:2002, gou:jas:2006, pal:2011}, and as a prior in Bayesian nonparametric statistics \citep{wal:etal:2007, fav:etal:2009, per:etal:2017}.

We study a model with a very general drift term in which equation \eqref{eq:SDE} takes the specific form
\begin{align}
\label{eq:WFSDE}
\dd X_t &= \frac{1}{2}\left[\tup(1-X_t) - \tdown X_t + \s X_t(1-X_t)\eta(X_t)\right]\dd t + \sqrt{X_t(1-X_t)} \dd W_t,
\end{align} 
where $\theta = (\tup, \tdown,\s) \in [0,\infty) \times [0,\infty) \times \bbR = \Theta$, and $\eta:[0,1]\to\bbR$ is a Lipschitz continuous function (assumed for identifiability of $\s$ not to be identically zero). Here $\gs(X_t)\dd W_t$ captures the effects of random mating while $\gm_\gq(X_t) \dd t$ captures the effects of recurrent mutation between two alternative alleles at rates $\tup/2$ and $\tdown/2$ (that is, between the allele whose frequency is currently $X_t$ and the other allele whose frequency is currently $1-X_t$), as well as the effects of selection with selection parameter $\s$. The function $\eta$ captures any frequency-dependent effects of selection. For example, \emph{genic} or \emph{haploid} selection has $\eta(x) = 1$ while a standard model of diploid selection has $\eta(x) = x + h(1-2x)$ with dominance $h$. It is well known that these assumptions satisfy the requirements for \eqref{eq:WFSDE} to have a unique, strong solution in the space of continuous paths $C([0,\infty),[0,1])$ \citep[e.g.][]{sat:1976:JMSJ}.

We first address the question of the mutual arrangement of a pair of Wright--Fisher diffusion path measures $\bPT{\gq_0}{T}$ and $\bPT{\gq}{T}$ corresponding to solutions of \eqref{eq:WFSDE} up to time $T$. A celebrated result due to \citet{daw:1978} \citep[see also][Ch.~7.6]{eth:2000} states that if $\gq_0 = (\tup,\tdown,0)$ and $\gq = (\tup,\tdown,\s)$ then $\bPT{\gq_0}{T}$ and $\bPT{\gq}{T}$, respectively the \emph{neutral} and \emph{non-neutral} models, are equivalent if $T$ is finite. 
However, the relationship between $\bPT{\gq}{T}$ and $\bPT{\gq_0}{T}$ when the \emph{mutation} parameters $\tup$ and $\tdown$ are also allowed to differ seems to be an open problem. To determine the relationship in this more general setting is a much more complicated task since the Novikov condition (equation \eqref{eq:novikov} below), the natural tool to prove equivalence when it is merely the selection parameter differing, and other conditions similar to Novikov's, no longer hold. Indeed our intuitions about the qualitative nature of sample paths of $\bPT{\gq}{T}$ immediately rules out equivalence in general, which we can see by recalling Feller's boundary classification:
\begin{itemize}
\item If $\tup = 0$ then the endpoint at 0 is exit (as in exit-not-entrance),
\item if $0 < \tup < 1$ then the endpoint at 0 is regular and instantaneously reflecting, and
\item if $\tup \geq 1$ then the endpoint at 0 is entrance (as in entrance-not-exit).
\end{itemize}
A similar classification applies for $\tdown$ and the endpoint at 1. (See \citet[Ch.~15.6, p.~239--241]{kar:tay:1981} for a derivation of these classifications when $\s=0$. It is straightforward to show, and expected in light of Dawson's result, that similar calculations with $\s\neq 0$ leave the boundary classifications unchanged.)
So for example we can immediately see that if $\tup_0 \geq 1$ and $\tup_1 < 1$ then $\bPT{\tup_1,\tdown,\s}{T}$ cannot be absolutely continuous with respect to $\bPT{\tup_0,\tdown,\s}{T}$ because the event that a sample path hits the 0-endpoint before time $T$ has positive probability under $\bPT{\tup_1,\tdown,\s}{T}$ but not under $\bPT{\tup_0,\tdown,\s}{T}$. The boundary classifications alone do not however provide a complete picture: for example, they do not tell us whether $\bPT{\tup_1,\tdown,\s}{T}$ is absolutely continuous with respect to $\bPT{\tup_0,\tdown,\s}{T}$ when both $\tup_0,\tup_1 \in (0,1)$. As we discover below, this relationship also fails as long as $\tup_0 \neq \tup_1$.

To address questions such as these we exploit the theory of \emph{separating times} \citep{che:uru:2006, mij:uru:2012:PTRF} in order to completely characterise the mutual arrangement of two Wright--Fisher diffusion path measures. Roughly, a mutual separating time is a stopping time before which two measures are equivalent and after which they are mutually singular. This theory has recently been extended to allow for boundary behaviour as rich as that of the Wright--Fisher diffusion \citep{cri:uru:2026}. We use this to derive the separating time for two Wright--Fisher diffusion measures and, as might have been predicted from the discussion above, show that they are closely related to the first hitting time of an endpoint of the state space. Although separating time theory, relying heavily on the concepts of the scale function and speed measure of a diffusion, is inherently one-dimensional, we can extend some of these results to multidimensional Wright--Fisher diffusions (indeed to the infinite-dimensional Fleming--Viot process), by exploiting its convenient projective properties.

We study in closer detail some sample path properties on hitting the boundary by looking at the case when the diffusion is initiated from the endpoint of its state space, $x_0 \in \{0,1\}$, even if that endpoint is of entrance type. Loosely, we show that two path measures differing in their mutation parameters are `instantaneously singular' on emergence from the boundary. To achieve this for a reflecting endpoint we derive a new zero-one-type law for the class of additive functionals constructed from negative-degree monomials of the diffusion. For an entrance boundary, on the other hand, we use a comparison argument. The appropriate comparison is not with a Brownian motion but with a \emph{squared Bessel process}. This has major advantages because the squared Bessel process has essentially the same boundary behaviour at 0 as the Wright--Fisher diffusion and is very well studied \citep[e.g.][Ch.~XI]{rev:yor:1999}. There exist numerous results for us to exploit when studying the squared Bessel process owing to, for example, its close relationship to the branching property and to the radial component of Brownian motion. The mutual arrangement of pairs of (squared) Bessel processes, and associated zero-one laws, have been studied by \citet{shi:wat:1973}, \citet{pit:yor:1981}, \citet{xue:1990}, \citet{ove:1998}, \citet{che:2001}, and \citet{che:uru:2006}, among others.

Finally, as an application we address the problem of inference by deriving the maximum likelihood estimator (MLE) for the three parameters $(\tup,\tdown,\s)$ in \eqref{eq:WFSDE} based on a continuously observed sample path $X^T$. The MLE for $\s$ alone is studied in detail by \citet{wat:1979} and \citet{san:etal:2022} but to our knowledge the expression for the joint estimator for both mutation and selection is new. Studying the subproblem of joint estimation of $(\tup,\tdown)$ in detail, I show that the estimator, after introducing a suitable adjustment for the possibility of separation in finite time, is strongly consistent for certain submodels, otherwise there is no consistent estimator. For certain submodels we also derive a central limit theorem. These properties are all closely related to the separating time of the underlying path measures, where on separation we find that one of the two mutation parameters `causing' the separation becomes discoverable and the error of the estimator for that parameter automatically collapses to 0.

This paper is structured as follows. In \sref{sec:prelim} we introduce some notation and recall the notion of the \emph{separating time} between two measures. \sref{sec:separating} contains a detailed study of the separating time for the Wright--Fisher diffusion; a complete characterization is given in \tref{tab:S}. \sref{sec:FV} provides some extensions of our results on separating times to the Fleming--Viot process, an infinite-dimensional analogue of the Wright--Fisher diffusion. In \sref{sec:bessel} we tackle the special case of initiation from the boundary, $x_0\in\{0,1\}$, which requires new techniques---in particular, comparison with a squared Bessel process whose many known properties we can exploit. Finally, in \sref{sec:inference} we study the problem of statistical inference from a continuously observed Wright--Fisher diffusion, deriving the maximum likelihood estimator and its asymptotic properties.

\section{Notation and preliminaries}
\label{sec:prelim}
\subsection{Separating times}
We first summarise the notion of (mutual) separating time for two probability measures as introduced by \citet{che:uru:2006}; see also \citet{mij:uru:2012:PTRF}. Although much of the following discussion holds for any pair of measures, for notational convenience we focus on the case where both come from a given parametric family $\{\bP{\gq}:\:\gq\in\gQ\}$.

Let $\bP{\gq_1}$ and $\bP{\gq_2}$ be two probability measures on a measurable space $(\Omega,\cF)$ with a right-continuous filtration $(\cF_t)_{t\in[0,\infty)}$. For a stopping time $\tau$ we define
\[
\cF_{\tau} \coloneq \left\{A \in \cF:\: A\cap\{\tau \leq t\} \in \cF_t \text{ for any }t\in[0,\infty)\right\}, 
\]
so that $\cF_\infty = \cF$. The restriction of $\bP{\gq}$ to $\cF_\tau$ is denoted $\bPT{\gq}{\tau}$.

It is convenient to augment the time axis $[0,\infty]$ with an additional point $\gd$ and to extend the total order on $[0,\infty]$ by setting $\infty < \gd$. The extra point will allow us to distinguish pairs of measures that separate ``at infinity'' (with separating time $\infty$) from those that ``never separate'' (separating time $\gd$).
\begin{definition}
An \emph{extended stopping time} is a map $\tau:\gO \to [0,\infty]\cup\{\gd\}$ such that $\{\tau \leq t\}\in \cF_t$ for any $t\in [0,\infty]$.
\end{definition}

\begin{definition}
A \emph{time separating $\bP{\gq_2}$ from $\bP{\gq_1}$} is an extended stopping time $S$ that satisfies
\begin{align*}
\bPT{\gq_2}{\tau} &\ll \bPT{\gq_1}{\tau} \text{ on the set }\{\tau < S\},\\
\bPT{\gq_2}{\tau} &\perp \bPT{\gq_1}{\tau} \text{ on the set }\{\tau \geq S\},
\end{align*}
for all stopping times $\tau$.
\end{definition}
A time separating $\bP{\gq_2}$ from $\bP{\gq_1}$ exists and is $\bP{\gq_2}$-a.s.~unique \citep[Cor.~2.1]{che:uru:2006}. 
\begin{definition}
A \emph{(mutual) separating time for $\bP{\gq_1}$ and $\bP{\gq_2}$} is an extended stopping time $S$ that satisfies
\begin{align*}
\bPT{\gq_2}{\tau} &\sim \bPT{\gq_1}{\tau} \text{ on the set }\{\tau < S\},\\
\bPT{\gq_2}{\tau} &\perp \bPT{\gq_1}{\tau} \text{ on the set }\{\tau \geq S\},
\end{align*}
for all stopping times $\tau$.
\end{definition}
A mutual separating time for $\bP{\gq_1}$ and $\bP{\gq_2}$ exists and is $\bP{\gq_1},\bP{\gq_2}$-a.s.~unique \citep[Thm.~2.1]{che:uru:2006}. It is also clearly a time separating $\bP{\gq_2}$ from $\bP{\gq_1}$.

The utility of studying the separating time $S$ in order to find the mutual arrangement of two measures comes from the following equivalences \citep[Lemma 2.1]{che:uru:2006}:
\begin{itemize}
\item $\bP{\gq_2}$ and $\bP{\gq_1}$ are \emph{equivalent}, $\bP{\gq_2}\sim\bP{\gq_1}$, if and only if $S=\gd$ holds $\bP{\gq_1},\bP{\gq_2}$-a.s.
\item $\bP{\gq_2}$ is \emph{absolutely continuous} with respect to $\bP{\gq_1}$, $\bP{\gq_2} \ll \bP{\gq_1}$, if and only if $S = \gd$ holds $\bP{\gq_2}$-a.s.
\item $\bP{\gq_2}$ is \emph{locally equivalent}, that is $\bPT{\gq_2}{T} \sim \bPT{\gq_1}{T}$ for each $T\in[0,\infty)$ and written $\bP{\gq_2}\overset{\text{loc}}{\sim} \bP{\gq_1}$, if and only if $S \geq \infty$ holds $\bP{\gq_1},\bP{\gq_2}$-a.s.
\item $\bP{\gq_2}$ is \emph{locally absolutely continuous}, that is $\bPT{\gq_2}{T} \ll \bPT{\gq_1}{T}$ for each $T\in[0,\infty)$ and written $\bP{\gq_2}\overset{\text{loc}}{\ll} \bP{\gq_1}$, if and only if $S \geq \infty$ holds $\bP{\gq_2}$-a.s.
\item $\bP{\gq_2}$ and $\bP{\gq_1}$ are \emph{mutually singular}, $\bP{\gq_2} \perp \bP{\gq_1}$, if and only if $S \leq \infty$ holds $\bP{\gq_1},\bP{\gq_2}$-a.s.\ and if and only if $S \leq \infty$ holds $\bP{\gq_1}$-a.s.
\item $\bPT{\gq_2}{0}$ and $\bPT{\gq_1}{0}$ are \emph{mutually singular}, $\bPT{\gq_2}{0} \perp \bPT{\gq_1}{0}$, if and only if $S =0$ holds $\bP{\gq_1},\bP{\gq_2}$-a.s.\ and if and only if $S=0$ holds $\bP{\gq_1}$-a.s.
\end{itemize}

Our interest is in the separating time for a diffusion with state space $I$. Separating times are closely related to the first hitting time $T_z \coloneq \inf\{t\geq 0:\: X_t = z\}$ of certain `separating points' of the state space, and much of the difficulty arises in determining whether a boundary point of $I$ is separating. Recall that accessible (regular and exit) endpoints are included in the state space $I$ while inaccessible endpoints (entrance and natural) are not \citep[Ch.~16.7, pp.~365--370]{bre:1968}. For an interval state space $I\subseteq [-\infty,\infty]$ with interior $I^\circ$ and closure $\bar{I}$, the boundary is $\partial I = \bar{I}\setminus I^\circ$ whatever the endpoint classifications are.

Following \citet{che:uru:2006}, to account for the presence of $\gd$ we distinguish between $T_z$, which uses the convention $\inf\emptyset = \infty$, and $\overline{T}_z$, which uses the convention $\overline{\inf}\emptyset = \gd$. The idea is that if $z$ is a separating point then $S \leq \overline{T}_z$.

To define a separating point for a general diffusion on $I$ characterised by its scale function and speed measure is quite involved \citep{cri:uru:2026} and so we defer it to Appendix \ref{sec:separating-points}. The definition allows for diffusions with boundary behaviour of any classification. However, for a diffusion satisfying an It\^o SDE of the form \eqref{eq:SDE} with either inaccessible or absorbing boundaries, the definition of a non-separating interior point is more straightforward and can be recalled here.
\begin{definition} \label{def:non-separating}
For a diffusion of the form \eqref{eq:SDE} (up to the hitting time of the boundary), a point $z\in I^\circ$ is \emph{non-separating}, or \emph{good}, if there exists an open neighbourhood $N(z) \subset I$ of $z$ such that $\int_{N(z)}[(\gm_{\gq_1} - \gm_{\gq_0})^2/\gs^4](y) \dd y < \infty$, i.e.\ if $(\gm_{\gq_1} - \gm_{\gq_0})^2/\gs^4$ is locally integrable at $z$. Otherwise it is \emph{separating}.
\end{definition}
The definition is symmetric in $\gq_0$ and $\gq_1$. It is a slight specialisation of the one given in \citet{che:uru:2006} and \citet{mij:uru:2012:PTRF}; had we permitted the two measures to differ in their diffusion coefficients then we would additionally require that $\gs^2_{\gq_0} = \gs^2_{\gq_1}$ almost everywhere in $N(z)$. In Appendix \ref{sec:speed-and-scale} we give some intuition as to why it is the function $(\gm_{\gq_1} - \gm_{\gq_0})^2/\gs^4$ whose local integrability matters.

\subsection{Radon--Nikodym derivative}
Let $\Omega = C([0,\infty),I)$ denote the set of continuous functions from $[0,\infty)$ to $I$. Hereonwards we will work on the canonical filtered space $(\Omega, \cF, (\cF_t)_{t\in[0,\infty)})$ with $X_t(\omega) = \omega(t)$ and right-continuous filtration $\cF_t = \bigcap_{\ge > 0}\gs(X_s:\: s\in[0,t+\ge])$, $\cF \coloneq \bigvee_{t\in[0,\infty)} \cF_t$. Let $\bP{\gq_0}$ denote a probability measure corresponding to \eqref{eq:SDE} with parametrisation $\gq_0$. Of crucial importance in comparing solutions to \eqref{eq:SDE} for different $\theta$ is the Girsanov theorem and its extensions which tell us that the following conditions are equivalent:
\begin{enumerate}[(i)]
\item The stochastic exponential $Z = (Z_t)_{t\in[0,\infty)}$ given by
\begin{equation} \label{eq:Z_T}
Z_t = \begin{cases}
\lefteqn{\ds\exp\left(\int_0^t \uu(X_u) \dd W_u - \frac{1}{2}\int_0^t \uu^2(X_u) \dd u\right)}\\
\phantom{\exp\left(\int_0^t \uu(X_u) \dd W_u - \frac{1}{2} \right)} & \ds \text{if }t < \inf\left\{ s \in[0,\infty):\:\int_0^s \uu^2(X_u) \dd u = \infty\right\},\\
0 & \text{otherwise},
\end{cases}
\end{equation} 
where $\uu$ is the solution to
\begin{align}
\label{eq:b-int}
\gs(X_t)\uu(X_t) &= \gm_{\gq_1}(X_t) - \gm_{\gq_0}(X_t), & 0 &\leq t \leq T,
\end{align}
is a $\bP{\gq_0}$-martingale;
\item The measure $\bP{\gq_1}$ given by $\bPT{\gq_1}{T}(A) = \bbE_{\gq_0}[Z_T \bbI_A]$, $A\in\cF_T$, is well defined and in particular it satisfies, for each $T\in[0,\infty)$, 
\begin{equation}
\label{eq:equivalent}
\bPT{\gq_1}{T} \ll \bPT{\gq_0}{T}
\end{equation}
and
\begin{equation}
\label{eq:RND}
\frac{\dd \bPT{\gq_1}{T}}{\dd \bPT{\gq_0}{T}} = Z_T \qquad\text{ on }\cF_T;
\end{equation}
\item For each $T\in[0,\infty)$:
\begin{equation}
\label{eq:explode}
\bPT{\gq_1}{T}\left(\int_0^T \uu^2(X_t) \dd t < \infty\right) = 1.
\end{equation}
\end{enumerate}
See for example \citet{hob:rog:1998}, \citet{ruf:2015}, and \citet{cri:gla:2018}, who develop increasingly refined versions of this statement.

We will see below that the requirement that these conditions hold is in general too strong. It is precisely the language of separating times that help us to understand how they can fail, with $\bPT{\gq_1}{T} \ll \bPT{\gq_0}{T}$ holding not for every $T$ but only up to a separating time $S$ whose distribution we aim to find. The conditions above correspond to the (restrictive) property that $\bP{\gq_1}(S \geq \infty) = 1$. 

Under the `Engelbert--Schmidt' conditions for SDEs, which impose some local integrability conditions relating to $\mu_\gq$ and $\sigma^2$, \citet{mij:uru:2012:PTRF} showed that (i) $\bP{\gq_1}(S \geq \infty) = 1$ is equivalent to $Z_T$'s being a $\bP{\gq_0}$-martingale, and (ii) $\bP{\gq_1}(S = \delta) = 1$ is equivalent to $Z_T$'s being a uniformly integrable (u.i.)\ $\bP{\gq_0}$-martingale. The first of these results was generalised beyond the Engelbert--Schmidt conditions by \citet{des:etal:2021}. 

To apply Girsanov's theorem, the goal is to verify whether $Z$ is a $\bP{\gq_0}$-martingale. Since $Z$ is a non-negative local martingale and hence a supermartingale, this task boils down to checking whether $\bbE_{\gq_0}(Z_T) = 1$, a statement which makes sense even for $T=\infty$ in which case $Z$ is a u.i.\ martingale and we can let $T=\infty$ in \eqref{eq:equivalent} and \eqref{eq:RND}. 
There has been much research activity in finding simple sufficient conditions for $Z$ to be a martingale; perhaps the most well known is \emph{Novikov's condition} which states that $Z$ is a $\bP{\gq}$-martingale (resp.~u.i.\ $\bP{\gq}$-martingale) if 
\begin{equation}
\label{eq:novikov}
\bbE_\gq\left[\exp\left(\frac{1}{2}\int_0^T \uu^2(X_t) \dd t\right)\right] < \infty
\end{equation}
for each $T \in [0,\infty)$ (resp.~for $T=\infty$). Novikov's condition is quite strong in the sense that it requires the `accumulated disagreement'
\begin{equation}\label{eq:H_T}
H_T \coloneq \int_0^T b^2(X_t) \dd t
\end{equation}
to have certain exponential moments for each $T$. It is possible to construct simple examples of processes for which \eqref{eq:novikov} fails yet the process is a martingale \citep[see for example][Ch.~VIII, Exercise 1.30, p.~335]{rev:yor:1999}.

\section{Separating time of the Wright--Fisher diffusion} \label{sec:separating}
We begin this section by justifying our claimed equivalence between the neutral and non-neutral models. \citet{daw:1978} studies a measure-valued setting which includes a Fleming--Viot process $(\mu_t:\: t\geq 0)$ with type space $E = \bbR^d$ and mutation operating as a Brownian motion on $E$. It is possible to replace the mutation operator with one that is parent-independent (see \eqref{eq:FV2} below), and his argument still holds. Then we can replace the type space with one such that $|E| = 2$ to recover a one-dimensional version of his result, as follows. In any case we can give a direct proof.
\begin{theorem}\label{thm:daw}
For the Wright--Fisher diffusion \eqref{eq:WFSDE} 
we have $\bPT{(\tup,\tdown,\s)}{T} \sim \bPT{(\tup,\tdown,0)}{T}$ for each $T \in [0,\infty)$, and
\begin{equation}
\label{eq:selection}
\frac{\dd \bPT{(\tup,\tdown,\s)}{T}}{\dd \bPT{(\tup,\tdown,0)}{T}} = \exp\left(\frac{\s}{2}\int_0^T \sqrt{X_t(1-X_t)}\eta(X_t) \dd W_t - \frac{\s^2}{8}\int_0^T X_t(1-X_t)\eta^2(X_t) \dd t\right).
\end{equation}
\end{theorem}
\begin{proof}
If $\gq_0 = (\tup,\tdown,0)$ and $\gq_1 = (\tup,\tdown,\s)$ then $\uu(x) = \frac{\s}{2}\sqrt{x(1-x)}\eta(x)$ solves \eqref{eq:b-int} and is bounded on $[0,1]$, by $\pm K$ say, since $\eta$ is continuous on $[0,1]$ and therefore bounded. Hence
\begin{equation*}
\bbE\left[\exp\left(\frac{1}{2}\int_0^T \uu^2(X_t) \dd t\right)\right] \leq \exp\left(\frac{1}{2}K^2 T\right) < \infty
\end{equation*}
and Novikov's condition \eqref{eq:novikov} holds under $\bPT{(\tup,\tdown,\s)}{T}$ for any $(\tup,\tdown,\s)$. (Similarly, boundedness of $\uu$ implies \eqref{eq:explode} under any $\bPT{(\tup,\tdown,\s)}{T}$.) The expression \eqref{eq:selection} and $\bPT{(\tup,\tdown,\s)}{T} \ll \bPT{(\tup,\tdown,0)}{T}$ follow by Girsanov's theorem. Finally, $\bPT{(\tup,\tdown,0)}{T} \ll \bPT{(\tup,\tdown,\s)}{T}$ follows by similar arguments, reversing the roles of $\gq_1$ and $\gq_0$.
\end{proof}
\begin{remark}
From a statistical perspective this result can be seen as disappointing. Given a sample path of finite length, even if observed continuously and without error, it is not possible to say with certainty whether selection is acting on the allele.
\end{remark}
\begin{remark}
As is shown below, \thmref{thm:daw} fails for $T=\infty$ if the model is ergodic ($\tup,\tdown > 0$). This is consistent with our intuition that given a sample path $X$ of infinite length, by the ergodic theorem we are effectively observing the speed measure from which we can identify $\s$ without error.
\end{remark}

Our next goal is to generalize \thmref{thm:daw} to allow the mutation parameters to differ in addition to the selection parameter by computing the separating time for two measures $\bP{\gq_0}$ and $\bP{\gq_1}$ associated with \eqref{eq:WFSDE}. Recall that the state space $I$ is an interval which always has interior $I^\circ = (0,1)$, closure $\bar{I} = [0,1]$, and boundary $\partial I = \bar{I}\setminus I^\circ = \{0,1\}$, but $I$ itself may be open, half-open, or closed in $\bbR$ depending on the accessibility of its boundaries: if $\tup < 1$ then $0\in I$ is accessible, otherwise $0\notin I$;  if $\tdown < 1$ then $1\in I$ is accessible, otherwise $1\notin I$. To emphasise this dependence we write $I_\gq$.
\begin{theorem}
\label{thm:CU2022}
Let $\bP{\gq_i}$, $i=0,1$, be the laws of a Wright--Fisher diffusion \eqref{eq:WFSDE} under respective parametrisations $\gq_0 = (\tup_0,\tdown_0,\s_0)$, $\gq_1 = (\tup_1,\tdown_1,\s_1)$, and suppose $x_0 \in I_{\gq_0}\cap I_{\gq_1}$. Then, under $\bP{\gq_1}$, the separating time $S$ for $(\bP{\gq_0},\bP{\gq_1})$ is as given in \tref{tab:S}.

Furthermore, in those cases for which $\bP{\gq_1}(S \geq \infty) = 1$, the density process $Z$ given by \eqref{eq:RND} for each $T\in[0,\infty)$ is a $\bP{\gq_0}$-martingale and $Z_t$ can be expressed in the form given by equation \eqref{eq:Z_T}
with $\uu$ the solution to \eqref{eq:b-int}. More precisely, $\uu(x)$ is given by
\begin{align*}
\uu(x) &= \frac{\mu_{\gq_1}(x) - \mu_{\gq_0}(x)}{\gs(x)} = \frac{\tup_1-\tup_0}{2}\sqrt{\frac{1-x}{x}} + \frac{\tdown_1-\tdown_0}{2}\sqrt{\frac{x}{1-x}} + \frac{\s_1-\s_0}{2}\sqrt{x(1-x)}\eta(x)
\end{align*}
for $x\in I^\circ$, and for $x \in \partial I$ the definition of $\uu(x)$ is extended via continuity ($\uu(x)$ is necessarily zero whenever this extension is needed; that is, when $x$ is an accessible endpoint under $\bP{\gq_1}$ and $S \geq \infty$ holds $\bP{\gq_1}$-a.s.).

If even further we have $\bP{\gq_1}(S = \delta) = 1$ then $Z$ is a u.i.\ $\bP{\gq_0}$-martingale.
\end{theorem}

\begin{table}
\caption{\label{tab:S}The separating points $A$ associated with Wright--Fisher path measures $\bP{(\tup_0,\tdown_0,\s_0)}$ and $\bP{(\tup_1,\tdown_1,\s_1)}$, and the separating time $S$ under $\bP{(\tup_1,\tdown_1,\s_1)}$. In this table we assume without loss of generality that $\tup_1 \leq \tdown_1$. In the first five columns, the appearance of a set should be interpreted as ``$\tup_1 \in \cdot$'', and so on. There is no restriction on $\s_1 \in \bbR$.}
\begin{tabular}{rrrrrrcc}
\hline
Case & $\tup_1$ & $\tdown_1$ & $\tup_0$ & $\tdown_0$ & $\s_0$ & $A$ & $S$\\
\hline
{\sc (i)} & $[0,\infty)$ & $[0,\infty)$ & $\tup_1$ & $\tdown_1$ & $\s_1$ & $\emptyset$ & $\gd$\\
{\sc (ii)} &$0$ & $(0,\infty)$ & $0$ & $\tdown_1$ & $\bbR\setminus\{\s_1\}$ & $\emptyset$ & $\gd$\\
{\sc (iii)} &$(0,\infty)$ & $0$ & $\tup_1$ & $0$ & $\bbR\setminus\{\s_1\}$ & $\emptyset$ & $\gd$\\
{\sc (iv)} &$(0,\infty)$ & $(0,\infty)$ & $\tup_1$ & $\tdown_1$ & $\bbR\setminus\{\s_1\}$ & $\emptyset$ & $\infty$\\
{\sc (v)} &$0$ & $0$ & $(0,\infty)$ & $0$ & $\bbR$ & $\{0\}$ & $\overline{T}_0$\\
{\sc (vi)} &$0$ & $[0,1)$ & $0$ & $[0,\infty)\setminus\{\tdown_1\}$ & $\bbR$ & $\{1\}$ & $\overline{T}_1$\\
{\sc (vii)} &$[0,1$) & $[0,1)$ & $[0,\infty)\setminus\{\tup_1\}$ & $[0,\infty)\setminus\{\tdown_1\}$ & $\bbR$ & $\{0,1\}$ & $T_0\wedge T_1$\\
{\sc (viii)} &$[0,1)$ & $(0,\infty)$ & $[0,\infty)\setminus\{\tup_1\}$ & $\tdown_1$ & $\bbR$ & $\{0\}$ & $T_0$\\
{\sc (ix)} &$[0,1)$ & $[1,\infty)$ & $[0,\infty)\setminus\{\tup_1\}$ & $[0,\infty)\setminus\{\tdown_1\}$ & $\bbR$ & $\{0,1\}$ & $T_0$\\
{\sc (x)} &$0$ & $[1,\infty)$ & $0$ & $[0,\infty)\setminus\{\tdown_1\}$ & $\bbR$ & $\{1\}$ & $\gd$\\
{\sc (xi)} &$(0,1)$ & $(0,1)$ & $\tup_1$ & $[0,\infty)\setminus\{\tdown_1\}$ & $\bbR$ & $\{1\}$ & $T_1$\\
{\sc (xii)} &$(0,\infty)$ & $[1,\infty)$ & $\tup_1$ & $[0,\infty)\setminus\{\tdown_1\}$ & $\bbR$ & $\{1\}$ & $\infty$\\
{\sc (xiii)} &$[1,\infty)$ & $[1,\infty)$ & $[0,\infty)\setminus\{\tup_1\}$ & $[0,\infty)\setminus\{\tdown_1\}$ & $\bbR$ & $\{0,1\}$ & $\infty$\\
{\sc (xiv)} &$[1,\infty)$ & $[1,\infty)$ & $[0,\infty)\setminus\{\tup_1\}$ & $\tdown_1$ & $\bbR$ & $\{0\}$ & $\infty$\\
\hline
\end{tabular}
\end{table}

\begin{remark}
Of course, the separating time $S$ under $\bP{\gq_0}$ rather than $\bP{\gq_1}$ can be obtained from \tref{tab:S} by interchanging the roles of $\gq_0$ and $\gq_1$.
\end{remark}

To prove \thmref{thm:CU2022} we first need to find the separating points of the diffusion. To that end we calculate its scale function and speed measure by standard means (e.g.\ \citet[Ch.~3.3]{eth:2011} or \citet[Ch.~VII, Exercise~3.20, p.~311]{rev:yor:1999}). The scale function $S_\gq:I\to\bbR$ associated with \eqref{eq:WFSDE} is:
\begin{equation}
\label{eq:scale}
S_{\gq}(x) = C_\gq\int_d^x y^{-\tup}(1-y)^{-\tdown}e^{-\s H(y)} \dd y, \qquad H(y) \coloneq \int_d^y \eta(z) \dd z,
\end{equation}
where $C_\gq > 0$ and $d\in (0,1)$ are constants we do not need to specify. Note that $H$ is continuous on $[0,1]$ since $\eta$ is. Next, the speed measure $m_\gq$ is, on $(0,1)$,
\begin{equation}
\label{eq:speed1}
m_\gq(\dd x) = C_\gq^{-1} x^{\tup-1}(1-x)^{\tdown-1}e^{\s H(x)} \dd x,
\end{equation}
with the boundary behaviour manifested in the conditions:
\begin{align}
m_{\gq}(\{0\}) &= \begin{cases}
\infty & \tup = 0,\\
0 & \tup > 0,\end{cases} &
m_{\gq}(\{1\}) &= \begin{cases}
\infty & \tdown = 0,\\
0 & \tdown > 0.\end{cases}\label{eq:speed2}
\end{align}
If $X$ is ergodic we will choose $C_\gq = m_\gq([0,1])$ in \eqref{eq:speed1} so that $m_\gq$ is the stationary measure.

\begin{lemma}\label{lem:separating-points}
For the Wright--Fisher diffusion \eqref{eq:WFSDE} under laws $\bP{\gq_0}$ and $\bP{\gq_1}$:
\begin{enumerate}[(a)]
\item If $x\in (0,1)$ then $x$ is non-separating.
\item The endpoint 0 is non-separating if and only if $\tup_0 = \tup_1 < 1$.
\item The endpoint 1 is non-separating if and only if $\tdown_0 = \tdown_1 < 1$.
\end{enumerate}
\end{lemma}

\begin{proof}
See Appendix \ref{sec:separating-points}.
\end{proof}

\begin{remark}
If $\tup_0,\,\tup_1 \geq 1$ then by convention we say that $0$ is a separating point even if $\tup_0 = \tup_1$. The intuition that separation will occur ``by time $T_0$'' is still valid when $\tup_0 = \tup_1 \geq 1$ yet $\gq_0 \neq \gq_1$ because here we know that $T_0 = \infty$ under either measure, and by time $\infty$ the two measures will have separated for other reasons (either by hitting a different separating point or by the speed measure becoming observable over an infinite time horizon).
\end{remark}

\begin{proof}[Proof of \thmref{thm:CU2022}]
By \lemmaref{lem:separating-points}, the set of separating points for the pair $(\bP{\gq_0},\bP{\gq_1})$ is $A = A_0 \cup A_1$, where
\begin{align*}
A_0 &= \begin{cases}
\emptyset & \tup_0 = \tup_1 < 1,\\
\{0\} & \text{otherwise},
\end{cases} &
A_1 &= \begin{cases}
\emptyset & \tdown_0 = \tdown_1 < 1,\\
\{1\} & \text{otherwise}.
\end{cases}
\end{align*}
Hence using Theorem 2.18 of \citet{cri:uru:2026}, we can write $S$ in the form
\[
S = \begin{cases}
\gd & \gq_0 = \gq_1,\\
U \wedge V \wedge R & \gq_0 \neq \gq_1,
\end{cases}
\]
($\bP{\gq_0}, \bP{\gq_1}$-a.s.)\ where, after some simplification,
\begin{align*}
U &= \begin{cases}
T_0 & 0 \in A \text{ and either } X_0 = 0 \text{ or }\liminf_{t\uparrow T_0} X_t = 0,\\
\gd & \text{otherwise},
\end{cases}\\
V &= \begin{cases}
T_1 & 1 \in A  \text{ and either } X_0 = 1 \text{ or } \limsup_{t\uparrow T_1} X_t = 1,\\
\gd & \text{otherwise}.
\end{cases}\\
R &= \begin{cases}
\infty & 0 < \tup_0 = \tup_1 < 1\text{ and } 0 < \tdown_0 = \tdown_1 < 1,\\
\gd & \text{otherwise}.
\end{cases}
\end{align*}
The expression for $S$ can be further simplified because we can give the distribution of $T_0$ and $T_1$ in more detail. In particular we appeal to the following properties of the diffusion:
\begin{subequations}
\begin{align}
\bP{(\tup,\tdown,\s )}(\overline{T}_0 \wedge \overline{T}_1 = T_0 \wedge T_1 < \infty) = 1\; &\Longleftrightarrow \;\tup\vee\tdown < 1,\label{eq:a}\\
\bP{(\tup,\tdown,\s )}(\overline{T}_0 = T_0 < \infty) = 1\; &\Longleftrightarrow \;\tup < 1,\,\tdown > 0,\label{eq:b}\\
\bP{(\tup,\tdown,\s )}(\overline{T}_1 = T_1 < \infty) = 1\; &\Longleftrightarrow  \;\tdown < 1,\,\tup > 0,\label{eq:c}\\
\bP{(\tup,\tdown,\s )}(T_0 = \infty) = 1\; &\Longleftrightarrow \;\tup \geq 1,\label{eq:d}\\
\bP{(\tup,\tdown,\s )}(T_1 = \infty) = 1\; &\Longleftrightarrow \;\tdown \geq 1,\label{eq:e}\\
\bP{(\tup,\tdown,\s )}\Big(\liminf_{t\uparrow T_0} X_t = 0\Big) = 1\; &\Longleftrightarrow \;\tdown > 0,\label{eq:f}\\
\bP{(\tup,\tdown,\s )}\Big(\limsup_{t\uparrow T_1} X_t = 1\Big) = 1\; &\Longleftrightarrow \;\tup > 0,\label{eq:g}
\end{align}
\end{subequations}
which follow easily from its boundary classifications and the fact that the diffusion is regular; see for example \citet[Thm.~33.15, p.~766]{kal:2021} which provides a comprehensive description (that theorem is for a diffusion on natural scale. It is a straightforward matter to transform $X^T$ to its natural scale and back again via the mapping $S_{\gq}$ in \eqref{eq:scale}.). \tref{tab:S} can then be constructed as follows. First we note that the first five columns encompass all possible combinations of $(\tup_1,\tdown_1,\s_1)$ and $(\tup_0,\tdown_0,\s_0)$, grouped into different cases each of which we address in turn. In the following list, each equality and inequality of random variables should be understood as holding $\bP{\gq_1}$-a.s.
\begin{enumerate}[{\sc (i)}]
\item Here $\gq_0 = \gq_1$ and we immediately obtain $S=\delta$.
\item Here $A = \emptyset$ so $U=V=\delta$, and one endpoint is absorbing so $R=\delta$. Hence $S=\delta$.\label{item:ii}
\item Similar to \ref{item:ii}.
\item Here $A = \emptyset$ so $U=V=\delta$, and both boundaries are non-separating and reflecting so $R=\infty$. Hence $S=R=\infty$.
\item Here $A=\{0\}$ so $V=\delta$, and $U=T_0$ if $\liminf_{t\uparrow T_0}X_t = 0$ and $U=\delta$ otherwise. This can be written concisely as $U=\overline{T}_0$ since $\liminf_{t\uparrow T_0}X_t = 0$ holds a.s.\ on $\{T_0 < \infty\}$ by sample path continuity and fails a.s.\ on $\{T_0 = \infty\}$, noting that the endpoint at 1 is absorbing and from \eqref{eq:a} must be hit in finite time on $\{T_0=\infty\}$. Neither endpoint is reflecting so $R=\delta$. Hence $S=U=\overline{T}_0$.\label{item:v}
\item Here $A=\{1\}$ so $U = \delta$, and $V=T_1$ if $\limsup_{t\uparrow T_1}X_t = 1$ and $V=\delta$ otherwise. This can be written as $V=\overline{T}_1$ similarly to \ref{item:v}. The endpoint 0 is not reflecting so $R=\delta$. Hence $S=V=\overline{T}_1$.
\item Here $A=\{0,1\}$ so $U=T_0$ if $\liminf_{t\uparrow T_0}X_t = 0$ and $U=\delta$ otherwise, and $V=T_1$ if $\limsup_{t\uparrow T_1}X_t = 1$ and $V=\delta$ otherwise. Property \eqref{eq:a} tells us that either $T_0 <\infty$ or $T_1 < \infty$ and hence either $U = T_0 <\infty$ or $V = T_1 < \infty$. Thus since $R \geq \infty$ we know that $S = U \wedge V = T_0 \wedge T_1$.
\item Here $A=\{0\}$ so $U=T_0$ if $\liminf_{t\uparrow T_0}X_t = 0$ and $U=\delta$ otherwise. From \eqref{eq:f} we know the former condition holds a.s.\ so $U=T_0$, and from \eqref{eq:b} we have $T_0 < \infty$. Since $V=\delta$ and $R \geq \infty$ we obtain $S = U = T_0$.
\item Here $A = \{0,1\}$ so $U=T_0$ if $\liminf_{t\uparrow T_0}X_t = 0$ and $U=\delta$ otherwise, and $V=T_1$ if $\limsup_{t\uparrow T_1}X_t = 1$ and $V=\delta$ otherwise. Property \eqref{eq:f} implies $U = T_0$ and properties \eqref{eq:b} and \eqref{eq:e} imply $T_0 < T_1=\infty$. Hence $V = \infty$ and, with $R \geq \infty$, we get $S = U = T_0$.
\item Here $A=\{1\}$ so $U=\delta$, and $V=T_1$ if $\limsup_{t\uparrow T_1}X_t = 1$ and $V=\delta$ otherwise. But by properties \eqref{eq:b} and \eqref{eq:e} we have $T_0 < T_1 = \infty$. Hence $\limsup_{t\uparrow T_1}X_t < 1$ since 0 is absorbing, so $V=\delta$. Neither endpoint is reflecting so $R=\delta$, and we arrive at $S = \delta$.
\item Here $A=\{1\}$ so $U=\delta$, and $V=T_1$ if $\limsup_{t\uparrow T_1}X_t = 1$ and $V=\delta$ otherwise. By property \eqref{eq:c} we know $T_1 < \infty$ and thus $\limsup_{t\uparrow T_1}X_t = 1$ by sample path continuity, so $V = T_1$. Since $\tdown_0 \neq \tdown_1$ we have $R=\delta$, and hence $S = V = T_1$.
\item Here $A=\{1\}$ so $U=\delta$, and $V=T_1$ if $\limsup_{t\uparrow T_1}X_t = 1$ and $V=\delta$ otherwise. By property \eqref{eq:g} we have $\limsup_{t\uparrow T_1}X_t = 1$, so $V = T_1$, and by property \eqref{eq:e} we have $T_1 = \infty$. The endpoint at 1 is not reflecting so $R=\delta$. Hence $S=V=T_1=\infty$.
\item Here $A=\{0,1\}$ so $U=T_0$ if $\liminf_{t\uparrow T_0}X_t = 0$ and $U=\delta$ otherwise, and $V=T_1$ if $\limsup_{t\uparrow T_1}X_t = 1$ and $V=\delta$ otherwise. By properties \eqref{eq:f} and \eqref{eq:g} we know that $\liminf_{t\uparrow T_0}X_t = 0$ and $\limsup_{t\uparrow T_1}X_t = 1$, so $U=T_0$ and $V=T_1$. By properties \eqref{eq:d} and \eqref{eq:e} we further have that $T_0 = T_1 = \infty$. Thus, noting that $R=\delta$, we have $S = U \wedge V = T_0 \wedge T_1 = \infty$.
\item Here $A=\{0\}$ so $V=\delta$, and $U=T_0$ if $\liminf_{t\uparrow T_0}X_t = 0$ and $U=\delta$ otherwise. By property \eqref{eq:f} we have $\liminf_{t\uparrow T_0}X_t = 0$ and so $U=T_0$. Further, $T_0 = \infty$ by property \eqref{eq:d}. With $R=\delta$ we obtain that $S = U = T_0 = \infty$.
\end{enumerate}

It remains to verify the claimed properties of the density process $Z$. First note that, by the fundamental theorem of calculus, the integral form for $S_{\gq_0}(x)$ in \eqref{eq:scale} says that $S_{\gq_0}$ is absolutely continuous at least on any compact subinterval of $I_{\gq_0}^\circ$. By Theorem 2.1 of \citet{cri:uru:2025}, this property of $S_{\gq_0}$ is equivalent to X exhibiting the \emph{representation property}: every local $(\cF_t)$-martingale $M = (M_t:\: t\in[0,\infty))$ can be written in the form
\[
M_t = M_0 + \int_0^t H_s \dd X_s^c,
\]
where $H \in L^2_{\text{loc}}(X)$ (in other words, $H$ is a predictable process for which there exists a localising sequence $(T_n) \uparrow \infty$ of stopping times such that each stopped process $H^{T_n}$ is square-integrable with respect to the quadratic variation process $\langle X \rangle$), and $X^c$ is the continuous local martingale part of $X$.

Now suppose $\bP{\gq_1}(S \geq \infty) = 1$, which we recall is the same as $\bP{\gq_1}\overset{\text{loc}}{\ll} \bP{\gq_0}$. Under this condition together with the fact that $X$ satisfies \eqref{eq:WFSDE} and is therefore a semimartingale, the density process $Z$ in \eqref{eq:RND} exists and is a $\bP{\gq_0}$-martingale \citep[Thm.~ III.3.4, p.~166]{jac:shi:2003}. Furthermore, by the representation property we can invoke Theorem III.5.19 of \citet[p.~194]{jac:shi:2003} to obtain an expression for the density process $Z$ in which $H$ in the representation property is given by \eqref{eq:H_T}; this recovers \eqref{eq:Z_T} as required. (To see that Theorem III.5.19 of \citet{jac:shi:2003} precisely yields \eqref{eq:Z_T}, their notation can be aligned with ours by setting the semimartingale characteristics to $(B_t,C_t) = (\int_0^t \mu_{\gq_0} (X_s) \dd s, \int_0^t \gs^2(X_s) \dd s)$, and $A_t = t$, $\beta = b/\gs$, $c = \gs^2$. Then $B'_t = \int_0^t \mu_{\gq_1} (X_s) \dd s$, $N_t = \int_0^t \uu(X_s)\dd W_s$, and $\gD = \cup_{n\in\mathbb{N}} [0,\inf\{t \geq 0:\: H_t \geq n\}] = [0,\inf\{t \geq 0:\: H_t = \infty\})$.) Note that $\bP{\gq_1}(\inf\{t \geq 0:\: H_t = \infty\} = \infty)=1$ so that under $\bP{\gq_1}$ the stochastic exponential in \eqref{eq:Z_T} is in fact defined for all $t \in [0,\infty)$.

Finally, as noted earlier, the condition $\bP{\gq_1}(S = \delta) = 1$ is the same as $\bP{\gq_1} \ll \bP{\gq_0}$. By Prop.~III.3.5 of \citet[p.~167]{jac:shi:2003}, this absolute continuity condition is equivalent to $Z$ being a u.i.\ $\bP{\gq_0}$-martingale; see also Cor.~2.23 of \cite{cri:uru:2026}.
\end{proof}

\begin{remark}
An alternative, albeit less direct, route to deriving the density process $Z$ would be to invoke Theorem 2.7 of \citet{des:etal:2021}, which generalises to diffusions that are possibly not semimartingales.
\end{remark}

\section{Fleming--Viot measures} \label{sec:FV}
As much of the preceding work relies heavily on our knowledge about the scale function and speed measure, it does not generalize readily to multidimensional diffusions. There is a natural multidimensional analogue of \eqref{eq:WFSDE} to describe the joint frequency of $K$, rather than two, alleles. By casting the process as a measure-valued diffusion, we can even pass straight to an infinite-dimensional analogue whose allele space can be an arbitrary (compact) Polish space $(E,\sE)$. This is the Fleming--Viot process $(\gm^\nu_t:\: t\geq 0)$ which takes values in $\cP(E)$, the set of Borel probability measures on $E$ with the topology of weak convergence. We will consider the class of neutral Fleming--Viot processes whose generator is of the form
\begin{align}\label{eq:FV1}
(\cL \varphi)(\gm) = {}& \frac{1}{2}\int_E\int_E [\gd_x (\dd y) - \gm(\dd y)] \frac{\gd^2 \varphi(\gm)}{\gd \gm(x)\gd\gm(y)}\mu(\dd x) + \int_E \cA\left(\frac{\gd\varphi(\gm)}{\gd \gm(\cdot)}\right)(x)\gm(\dd x)
\end{align}
where $\gd\varphi(\gm)/\gd\gm(x) = \lim_{\ge\to0+}[\varphi(\gm + \ge\gd_x) - \varphi(\gm)]/\ge$, $\gd_x$ is the Dirac measure at $x$, and
\begin{equation}\label{eq:FV2}
(\cA f)(x) = \frac{1}{2}\int_E [f(y) - f(x)]\nu(\dd y)
\end{equation}
is the generator for a Markov process describing the effects of (parent-independent) mutation, with $\nu$ a finite measure on $E$. The domain for $\cL$ is the set of bounded functions on $\cP(E)$ of the form
\[
\varphi(\gm) = F(\langle f_1,\gm\rangle,\dots, \langle f_k,\gm\rangle), \qquad \langle f,\gm\rangle \coloneq \int_E f\dd \gm,
\]
where $k \geq 1$, $f_1,\dots,f_k$ are in the domain $C(E)$ of $\cA$, and $F \in C^2(\bbR^k)$. See \citet{eth:kur:1993} for further details.

Although we cannot give a complete answer to the mutual arrangement of two Fleming--Viot path measures, a convenient projective property of the family given above allows us to detect certain times by which separation must have occurred. We denote the path measure associated with \eqref{eq:FV1}--\eqref{eq:FV2}, parametrised by $\nu$, by $\bbF\bbV_\nu$, and we denote the one-dimensional projection of this process onto some $B \in \sB(E)$ by $\gm^{\nu}(B) = (\gm^{\nu}_t(B):\: t\geq 0)$.

\begin{theorem}\label{thm:FV}
Let $\sS$ be the separating time associated with $(\bbF\bbV_{\nu_1}, \bbF\bbV_{\nu_0})$. Then for any $B \in \sB(E)$:
\[
\sS \leq S_B
\]
($\bbF\bbV_{\nu_1}$, $\bbF\bbV_{\nu_0}$-a.s.), where $S_B$ is the separating time associated with $(\gm^{\nu_1}(B), \gm^{\nu_0}(B))$. Furthermore, $S_B$ is given by \tref{tab:S} when we identify $\tup_1 = \nu_1(B)$, $\tdown_1 = \nu_1(E\setminus B)$, $\tup_0 = \nu_0(B)$, $\tdown_0= \nu_0(E\setminus B)$, and $\s_1 = 0 = \s_0$.
\end{theorem}

\begin{proof}
The first claim follows from the obvious fact that $(\bbF\bbV_{\nu_1}, \bbF\bbV_{\nu_0})$ must separate if one of its projections does. The second claim follows from the fact that $\mu^{\nu}(B)$ is a Wright--Fisher diffusion of the form \eqref{eq:WFSDE} with parameters $\gq = (\nu(B), \nu(E\setminus B), 0)$ \citep{eth:kur:1993}.
\end{proof}

To use \thmref{thm:FV} to try to detect separation from a realization $\omega$, one should therefore look for a region $B \in \sB(E)$ for which $\nu_1(B) \neq \nu_0(B)$ and $\mu^{\nu}_t(B)(\omega) = 0$ for some $t$. From a statistical perspective, $\nu(B)/2$ is the rate at which an allele mutates to a type in $B$, so we are seeking a region of the type space for which the total rate of mutation incoming to $B$ differ between the two measures, \emph{and} we wish to observe a sample path as the types in $B$ collectively go extinct, since it is on the road to extinction that the ``signal to noise'' ratio of the diffusion explodes. Through a few versatile choices of type space $E$, the Fleming--Viot process has many genetic applications \citep[Sec.~9]{eth:kur:1993}, and pure jump operators of the form \eqref{eq:FV2} can be used to model, for example, the infinitely-many-alleles mutation model. There is interest in performing statistical inference with the Fleming--Viot model \citep{asc:etal:2023, asc:etal:2024}, though these applications typically assume a discretely, rather than continuously, observed diffusion. In that setting, the observation frequency is too low for the separation of measures to be detectable. In the statistics literature, the problems that arise from a discretely observed diffusion on $\bbR^d$ when a parameter affects the \emph{diffusion} coefficient are well known \citep{rob:str:2001}: a Markov chain Monte Carlo algorithm which updates this parameter ($\phi$ say, where $\sigma(\cdot) = \sigma_\phi(\cdot)$) mixes poorly, and the mixing \emph{worsens} as the observation frequency increases because the mutual singularity between the path measures for two different parameter values becomes more evident (effectively $\phi$ is observable via the quadratic variation of the path). In our context, the presence of a boundary means we can expect a similar phenomenon but for a parameter that affects only the \emph{drift} coefficient.

Another non-equivalence statement for a measure-valued diffusion is given by \citet[Proposition 5.2]{daw:1978}, which studies the Dawson--Watanabe superprocess. That result also uses a convenient projective property, namely that the total mass of a Dawson--Watanabe process is a continuous-branching-with-immigration (CBI) diffusion. Other projection arguments, allowing statements about the mutual arrangement of multidimensional processes to be deduced from their projections, are given by \citet{cri:gla:2018} and \citet{cri:2021} for processes of a certain `partially radial' type.

\section{A closer look at boundary behaviour} \label{sec:bessel}
If $0$ is of entrance type then it is not inside the state space and therefore $x_0 = 0$ is not covered by \thmref{thm:CU2022} (similarly for the endpoint at $1$). It is, however, valid to initiate a diffusion from an entrance boundary; the sample path instantaneously enters $I$ and never returns to that endpoint. This idea is relatively well studied for the Bessel process \citep{pit:yor:1982}, though less so for the Wright--Fisher diffusion \citep[but see][]{sch:etal:2016, gri:etal:2018}.  A complete classification of the mutual relationship between two Bessel process path measures, allowing for the possibility that $x_0 = 0$ and either law makes $0$ an entrance boundary, is given by \citet{che:uru:2006}. Let $(\rho_t:\: t\geq 0)$ denote a Bessel process with $\rho_0 = 0$ and its corresponding path measure restricted to $\cF_T$ by $\bbB^T_\gk$, where $\gk \in [0,\infty)$ is its `dimension' parameter. Corollary 4.1 of \citet{che:uru:2006} shows that $\bbB^0_{\gk_0} \perp \bbB^0_{\gk_1}$ if $\gk_0 \neq \gk_1$.

The goal of this section is to conduct a similar study for the Wright--Fisher diffusion. Compared to the Bessel process, this is a much less straightforward endeavour because the arguments used to compare two path measures $(\bbB_{\gk_0},\bbB_{\gk_1})$ on $\cF_{0}$ rely heavily on the time-reversal property $\left(t\rho_{1/t}:\: t\geq 0\right) \overset{d}{=} (\rho_t :\: t \geq 0)$ \citep{shi:wat:1973}, and thus the behaviour of $\rho_t$ near $t=0$ can be obtained by studying its limiting behaviour as $t\to\infty$. This indispensable identity is, unfortunately, not available to us here.

\subsection{Reflecting endpoint} \label{sec:reflecting}
Although a reflecting endpoint is covered by \thmref{thm:CU2022}, we gain additional qualitative insight by a more direct argument, specifically via an application of the occupation times formula \citep[Ch.~VI, Cor.~1.6, p.~224]{rev:yor:1999}.  This approach will yield a new zero-one law for additive functionals of the Wright--Fisher diffusion.

We first recall the transformation taking a diffusion $X$ satisfying \eqref{eq:SDE} to a Brownian motion \citep[see for example][Thm.~V.47.1, p.~277]{rog:wil:2000:II}: there exists a Brownian motion $W$ with $W_0 = S_\gq(x_0)$ such that $X$ on its natural scale can be expressed via a time change of $W$:
\begin{equation}\label{eq:speed-and-scale}
Y_t \coloneq S_\gq(X_t) = W_{\lambda(t)},
\end{equation}
where $S_\gq:I\to \bbR$ is the scale function associated with $X$ and
\[
\lambda(t) \coloneq \inf\{s \in[0,\infty):\: \cA_s > t\}
\]
is the right-continuous inverse to the additive functional
\[
\cA_t \coloneq \int L^{y}_t(W) m_Y(\dd y) = \int_0^t m_Y(W_s) \dd s;
\]
here $(L^z_t(Z): z\in\bbR, t \in [0,\infty))$ is the semimartingale local time of a process $Z$, and $m_Y(\dd y) = m_Y(y)\dd y$ is the speed measure of $Y  = (Y_t:\: t\in[0,\infty))$. The latter process satisfies 
\begin{equation}
\label{eq:Y}
Y_t = S_{\gq}(x_0) + \int_0^t \gs_Y(Y_s) \dd W_s,
\end{equation}
and since 
\begin{equation}
\label{eq:mY}
m_Y(\dd y) = \gs^{-2}_Y(y) \dd y
\end{equation}
(\citet[Thm.~16.83, p.~387]{bre:1968}, \citet[Sec.~V.47.27, p.~283]{rog:wil:2000:II}), 
$\lambda(t)$ is given by the solution to
\begin{equation}
\label{eq:gamma}
\int_0^{\lambda(t)} \frac{1}{\gs^2_Y(W_s)} \dd s = t.
\end{equation}
By It\^o's lemma, $Y$ has diffusion coefficient
\begin{align}
\gs_Y(Y_t) &= (S_{\gq}'\gs)\circ S_{\gq}^{-1}(Y_t) = \gs(X_t)S_\gq'(X_t). \label{eq:sigmaY}
\end{align}
(For a diffusion with a reflecting endpoint it is perhaps more natural to express $Y$ in \eqref{eq:speed-and-scale} in terms of a \emph{reflecting} Brownian motion, though this is not necessary since the latter is itself a time change of $W$ \citep[Sec.~V.47.25, p.~283]{rog:wil:2000:II}.)

Let $f:\bar{I}\to[0,\infty)$ be a Borel measurable function and define $g = f \circ S_\gq^{-1}$. Applying the occupation times formula to $Y$, we find
\begin{align*}
\int_0^t f(X_s) \dd s &= \int_0^t g(Y_s) \dd s\\
&= \int_{S_{\gq}(l)}^{S_{\gq}(r)} g(y) L_t^y(Y) \sigma^{-2}_Y(y) \dd y\\ 
&= \int_{S_{\gq}(l)}^{S_{\gq}(r)} g(y) L_{\lambda(t)}^y(W) \sigma^{-2}_Y(y) \dd y &\text{by \eqref{eq:speed-and-scale},}\notag\\
&= \int_l^r \frac{f(x)}{\sigma^{2}(x)S_\gq'(x)} L_{\lambda(t)}^{S_\gq(x)}(W) \dd x & \text{by \eqref{eq:sigmaY} and using $y=S_\gq(x)$}.
\end{align*}

Now let us return to the Wright--Fisher diffusion, and suppose $\tup \in (0,1)$ in order to make the endpoint at 0 reflecting under $\bP{\gq}$. The process $Y_t = S_\gq(X_t)$ with scale function \eqref{eq:scale} satisfies \eqref{eq:Y} with
\begin{align}
\label{eq:sigmaY-WF}
\gs_Y(y) &= C_\gq\sqrt{x(1-x)}x^{-\tup}(1-x)^{-\tdown}e^{-\s H(x)}, & x &= S^{-1}_\gq(y),
\end{align}
and the occupation times formula becomes
\begin{equation}\label{eq:OTF}
\int_0^t f(X_s) \dd s = \int_0^1 f(x)x^{\tup - 1}(1-x)^{\tdown - 1}e^{\s H(x)} L_{\lambda(t)}^{S_\gq(x)}(W) \dd x.
\end{equation}
Our interest is in whether this integral is finite when $x_0 = 0$. To avoid having to consider the influence of the endpoint at 1 we will apply the formula to a function whose support is contained in $[0,\Delta]$ for a fixed $\Delta \in (0,1)$.  Note that (i) $(1-x)^{\tdown - 1}e^{\s H(x)}$ is continuous on $[0,\Delta]$, and (ii) $L_{\lambda(t)}^{y}(W)$ is a.s.-c\`adl\`ag in $y$ \citep[Ch.~VI, Thm.~1.7, p.~225]{rev:yor:1999} and therefore by continuity of $S_\gq$ the function $x \mapsto L_{\lambda(t)}^{S_\gq(x)}(W)$ is a.s.-c\`adl\`ag too. Thus if we suppose the function $h(x) \coloneq f(x)x^{\tup -1}$ is integrable on $[0,\Delta]$ then we may conclude from \eqref{eq:OTF} that
\[
\int_0^t f(X_s) \dd s < \infty. 
\]
On the other hand, $t > 0$ implies from \eqref{eq:gamma} that $\lambda(t) > 0$ and hence that $L_{\lambda(t)}^{S_\gq(0)}(W) > 0$ a.s.\ \citep[Problem 3.6.13, p.~209]{kar:shr:1998}, so there exists $\varepsilon > 0$ such that $x \mapsto L_{\lambda(t)}^{S_\gq(x)}(W) > 0$ is bounded away from 0 on $[0,\varepsilon]$. Thus if we suppose $h(\cdot)$ is \emph{not} locally integrable at 0---that is, for any $\varepsilon' > 0$ the function $h\bbI_{[0,\varepsilon']}$ is not integrable---then it follows from \eqref{eq:OTF} that
\[
\int_0^t f(X_s) \dd s = \infty \qquad \text{for any }t > 0.
\]

In summary, we have shown the following.

\begin{theorem} \label{thm:reflecting-f}
Let $X$ be a Wright--Fisher diffusion satisfying \eqref{eq:WFSDE} with $x_0 = 0$ and $\tup \in (0,1)$, and let $f:[0,1]\to[0,\infty)$ be a Borel measurable function with support contained in $[0,\Delta]$.
\begin{enumerate}
\item If $\Delta \in (0,1)$ and $f(x)x^{\tup -1}$ is integrable on $[0,\Delta]$ then $\bP{\gq}$-a.s.:
\[
\int_0^t f(X_s) \dd s < \infty. 
\]
\item If $f(x)x^{\tup -1}$ is not locally integrable at 0 then for each $t > 0$, $\bP{\gq}$-a.s.,
\[
\int_0^t f(X_s) \dd s = \infty.
\]
\end{enumerate}
\end{theorem}

\begin{corollary} \label{cor:reflecting-kappa}
Fix $\gk \geq 0$. Then, for each $t> 0$, $\bP{\gq}$-a.s.:
\begin{align}
\label{eq:separate}
\int_0^t \frac{1}{X_s^\gk} \dd s \phantom{=}\begin{cases}
< \infty, & 0 \leq \gk < \tup,\\
= \infty, & \gk \geq \tup.
\end{cases}
\end{align}
\end{corollary}
\begin{proof}[Proof of Corollary~\ref{cor:reflecting-kappa}]
Choose $f(x) = x^{-\kappa}\bbI_{(0,1]}(x)$ in \thmref{thm:reflecting-f}, and note that the restriction to $\Delta < 1$ can be relaxed in this case, since $f$ is continuous on $[\Delta, 1]$.
\end{proof}

\begin{corollary}\label{cor:reflecting-singular}
Let $X$ be as in \thmref{thm:reflecting-f}, and suppose $\tup_1 \in [0,1)$ with $\tup_1\neq \tup$. Then $\bPT{\gq}{0} \perp \bPT{\gq_1}{0}$.
\end{corollary}
\begin{proof}
First suppose $\tup_1 > 0$ and then without loss of generality that $\tup < \tup_1$. Choose $\gk$ such that $\tup < \gk < \tup_1$, and let $E_t \coloneq \{\int_0^t X_s^{-\kappa} \dd s = \infty\}$. By Corollary~\ref{cor:reflecting-kappa}, we know that for each $t > 0$, $\bP{\gq}(E_t) = 1$ while $\bP{\gq_1}(E_t) = 0$. Moreover, $E \coloneq \bigcap_{n\in\bbN} E_{1/n}$ is $\cF_0$-measurable on the right-continuous filtration and $\bP{\gq}(E) = 1$ whereas $\bP{\gq_1}(E) = 0$. Hence $\bPT{\gq}{0} \perp \bPT{\gq_1}{0}$.

Finally suppose $\tup_1 = 0$. Then clearly
\begin{align*}
\bP{\gq}\left(X_t = 0 \; \forall u\in [0,t]\right) &= 0, & \bP{\gq_1}\left(X_t = 0 \; \forall u\in [0,t]\right) &= 1,
\end{align*}
for each $t > 0$, so again $\bPT{\gq}{0} \perp \bPT{\gq_1}{0}$ by a similar argument.
\end{proof}

\subsection{Entrance boundary}
To summarise the previous subsection, our strategy in showing mutual singularity was to use a transformation in space and time to express $X$ in terms of a Brownian motion and then to invoke a zero-one law for certain additive functionals of Brownian motion. The strategy fails if $0$ is of entrance type, $\tup \geq 1$, since then $S_{\gq}(0) = -\infty$ and we would be forced to take $W_0 = -\infty$. We must adapt the argument in a different direction: the key idea is for the Brownian motion to be replaced with a \emph{squared Bessel process} \citep[Ch.~XI]{rev:yor:1999}. Indeed one might suspect that this is a natural candidate since either a simple spatial transformation $Y_t = 4X_t$ or a time transformation $Y_t = X_{4t}$ yields a new diffusion $Y$ with drift and diffusion coefficients satisfying 
\begin{align}
\label{eq:Bessel-asymptotic}
\mu_Y(y) &\sim 2\tup, &
\gs_Y(y) &\sim 2\sqrt{y},
\end{align}
as $y\downarrow 0$, asymptotically equivalent to a squared Bessel process, which solves
\begin{align}
\label{eq:BESQ}
\dd Z_t &= \gk\dd t + 2\sqrt{Z_t} \dd W_t, & Z_0 &= z_0 \in [0,\infty).
\end{align}
Here $\gk \in [0,\infty)$ is the `dimension' parameter. Recall that we denoted the path measure for a Bessel process $(\rho_t:\: t\geq 0)$ initiated from $\rho_0 = 0$ by $\bbB_\gk$; under this measure we have $(Z_t :\: t\geq 0) \overset{d}{=} (\rho^2_t :\: t\geq 0)$.

So at the 0 endpoint, we should expect the relationship between two Wright--Fisher diffusion measures to be analogous to a pair of squared Bessel process measures, with $2\tup$ playing the role of $\gk$. Indeed by identifying the parameters in this way the boundary classifications for the Wright--Fisher diffusion quoted in the introduction agree with those of the squared Bessel process:
\begin{itemize}
\item If $\gk = 0$ then the endpoint at 0 is exit,
\item if $0 < \gk < 2$ then the endpoint at 0 is regular and instantaneously reflecting, and
\item if $\gk \geq 2$ then the endpoint at 0 is entrance.
\end{itemize}

Following the reasoning of \sref{sec:reflecting}, if $1 \leq \tup_0 < \tup_1$ then we should choose a squared Bessel process with parameter $\gk$ such that $2\tup_0 < \gk < 2\tup_1$ and then hope to find an $\cF_0$-measurable `bisecting' event, what \citet{daw:1968} calls a \emph{germ field singularity}, analogous to \eqref{eq:separate} but for a squared Bessel process. Fortunately, such an event can be constructed from the following.

\begin{proposition}[\citet{che:2001}, p.~208]
\label{prop:cherny}
If $(\rho^2_t:\: t\geq 0)$ is a squared Bessel process then, $\bbB_\gk$-a.s.\ for $\gk \geq 2$ and for any $n\in\bbN$,
\begin{equation}
\label{eq:Bessel-separate}
\frac{1}{\log t}\int_{1/t}^{1/n} \frac{1}{\rho_s^{2}} \dd s \overset{t\to\infty}{\longrightarrow} \frac{1}{\gk - 2}.
\end{equation}
\end{proposition}

\begin{theorem} \label{thm:entrance}
Let $X$ be a Wright--Fisher diffusion satisfying \eqref{eq:WFSDE} with $x_0 = 0$ and $\tup \geq 1$, and suppose $\tup_1 \geq 1$ with $\tup_1 \neq \tup$. Then $\bPT{\gq}{0} \perp \bPT{\gq_1}{0}$.
\end{theorem}

\begin{proof}
Let $\psi: [0,1] \to [0,\pi^2]$ be given by
\[
\psi(x) = [\cos^{-1}(1-2x)]^2,
\]
and consider the process $Y$ defined as $Y_t = \psi(X_t)$. Using It\^o's lemma we find that, at least up to a stopping time $\inf\{t \geq 0:\: Y_t = \pi^2\}$, $Y$ has drift and diffusion coefficients 
\begin{align*}
{\mu}_Y(y) = {}& \left[\tup(1+\cos\sqrt{y}) - \tdown(1-\cos\sqrt{y}) + \frac{\s}{2}\widetilde\eta(y)\sin^2\sqrt{y}\right]\frac{\sqrt{y}}{\sin\sqrt{y}}\\
& {}+ 1-\sqrt{y}\cot\sqrt{y}, & y&\in (0,\pi^2),\\
\sigma_Y(y) = {} & 2\sqrt{y},
\end{align*}
with $\widetilde{\eta}(y) \coloneq \eta \circ \psi^{-1}(y) = \eta((1-\cos{\sqrt{y}})/2)$ and ${\mu}_Y(0) \coloneq \mu_Y(0+) = 2\tup$. 

We may assume without loss of generality that $\tup < \tup_1$. Choose $\gk$, $\gk_1 > 2$ such that $2\tup < \gk < \gk_1 < 2\tup_1$. Since $\mu_Y$ is continuous and $\mu_Y(0) = 2\tup$, there is a neighbourhood $[0,\varepsilon)$ of $0$, $0 < \varepsilon < \pi^2$, such that
\[
\mu_Y(y) < \gk, \qquad y \in [0,\varepsilon),
\]
and since we have engineered $Y$ to have the same volatility as a squared Bessel process $(\rho^2_t:\: t\geq 0)$, we can invoke a comparison theorem \citep[see Thm.~1.1 and Remark 1.1 of][]{ike:wat:1977} to construct a probability space on which $Y$ is governed via $\psi(\cdot)$ by $\bP{\gq}$, $\rho$ is governed by $\bbB_\gk$, and
\begin{equation}
\label{eq:comparison}
\bbP(Y_t \leq \rho^2_t\text{ for all }t\in[0,\gt_\varepsilon]) = 1.
\end{equation}
Here, $\gt_\varepsilon = \inf\{t \geq 0:\: \rho^2_t = \varepsilon\}$ and clearly $\bbP(\gt_\varepsilon > 0) = 1$.

Next introduce the events
\begin{align*}
A^\gk_n &\coloneq \left\{\liminf_{t\to \infty}\frac{1}{\log t}\int_{1/t}^{1/n} Y_s^{-1} \dd s \geq \frac{1}{\gk - 2}\right\},\\
R^\gk_n &\coloneq \left\{\lim_{t\to \infty}\frac{1}{\log t}\int_{1/t}^{1/n} \rho_s^{-2} \dd s = \frac{1}{\gk - 2}\right\}.
\end{align*}
Observe that
\begin{align}
\bbP(A^\gk_n) &\geq \bbP(A^\gk_n \cap \{\gt_\varepsilon \geq 1/n\})
\geq \bbP(R^\gk_n \cap \{\gt_\varepsilon \geq 1/n\})
= \bbP(\{\gt_\varepsilon > 1/n\})
\overset{n\to\infty}{\longrightarrow} 1, \label{eq:A_n}
\end{align}
with the second inequality following from the coupling \eqref{eq:comparison} on the set $\{\gt_\varepsilon \geq 1/n\}$ and the last equality following from \propref{prop:cherny}. 
Combining \eqref{eq:A_n} with the fact that $A^\gk_n$ is a nested sequence in the sense that $\bbP(A^\gk_n \cap A^\gk_{n+1}) = \bbP(A^\gk_n)$, we have that $A^\gk_n$ occurs eventually: if $A^\gk \coloneq \liminf_{n\to\infty} A^\gk_n$ then
\[
\bbP(A^\gk) = \bbP\left(\bigcup_{m}\bigcap_{n\geq m} A^\gk_n\right) = \bbP\left(\bigcup_m A^\gk_m\right) = 1.
 \]
 Furthermore, $A^\gk$ is $\cF_0$-measurable on the right-continuous filtration (noting that $\psi(\cdot)$ is bijective so that $A^\gk_n$ can be expressed in terms of $(X_s:\: s\in[0,1/n])$ rather than $(Y_s:\:s\in[0,1/n])$). In summary, we have shown
 \begin{equation}
 \label{eq:entrance-singular1}
 \bPT{\gq}{0}(A^\gk) = 1.
 \end{equation}
 By a similar argument (reversing the inequality in \eqref{eq:comparison}) we can show
 \begin{equation}
\label{eq:entrance-singular2}
\bPT{\gq_1}{0}(B^{\gk_1}) = 1,
 \end{equation}
 where $B^{\gk_1} = \liminf_{n\to\infty} B_n^{\gk_1}$ and
 \[
 B^{\gk_1}_n \coloneq \left\{\limsup_{t\to \infty}\frac{1}{\log t}\int_{1/t}^{1/n} Y_s^{-1} \dd s \leq \frac{1}{\gk_1 - 2}\right\}.
\]
 From \eqref{eq:entrance-singular1}--\eqref{eq:entrance-singular2} and the fact that $A^\gk \cap B^{\gk_1} = \emptyset$, we conclude $\bPT{\gq}{0} \perp \bPT{\gq_1}{0}$.
\end{proof}

\begin{remark}
The transformation $\psi$ (actually $\psi/4$) has been used to study the Kolmogorov partial differential equations associated with the Wright--Fisher diffusion (\citet[eq.~(3.3)]{che:str:2010}; see also \citet[Sec.~9]{eps:maz:2010}).
\end{remark}
\begin{remark}
Key to the argument above is to find a `signature' event $A^\gk$ that both reveals the underlying parameter $\gk$ and is $\cF_0$-measurable. This is possible for $\gk \geq 2$ from the result of \citet{che:2001} quoted above as \propref{prop:cherny}. A different result, Corollary 2.2 of \citet{che:2001}, offers a signature event for a Bessel process with a \emph{reflecting} endpoint, $0 < \gk < 2$, which provides an alternative route to proving our Corollary \ref{cor:reflecting-singular}.
\end{remark}

Combining Corollary \ref{cor:reflecting-singular} and \thmref{thm:entrance}, and extending the arguments a little further, we summarise this section as follows.

\begin{theorem} \label{thm:all}
Let $X$ be a Wright--Fisher diffusion satisfying \eqref{eq:WFSDE} and suppose either of the following holds:
\begin{enumerate}[(i)]
\item $x_0 = 0$ and $\tup \neq \tup_1$, or
\item $x_0 = 1$ and $\tdown \neq \tdown_1$.
\end{enumerate}
Then $\bPT{\gq}{0} \perp \bPT{\gq_1}{0}$.
\end{theorem}

\begin{proof}
In case (i) the only combination of $\tup$ and $\tup_1$ not already covered by Corollary \ref{cor:reflecting-singular} and \thmref{thm:entrance} is $\tup < 1$ and $\tup_1 \geq 1$ (and vice versa). But here it is straightforward to find a germ field singularity; for example we know
\begin{align*}
\bP{\gq}\left(\exists u \in (0,t]:\: X_u = 0\right) &= 1, & \bP{\gq_1}\left(\exists u \in (0,t]:\: X_u = 0\right) &= 0,
\end{align*}
for each $t > 0$. The first claim is inherited from Brownian motion's property that $t=0$ is not an isolated point of $\{t \geq 0:\: W_t = 0\}$ \citep[Thm.~12.35, p.~267]{bre:1968}, and the second claim follows from inaccessibility of the endpoint at $0$. Thus $\bPT{\gq}{0} \perp \bPT{\gq_1}{0}$ for all combinations of $\tup$ and $\tup_1$.

Case (ii) follows by applying all the same arguments to $1-X_t$.
\end{proof}

\section{Parameter estimation}\label{sec:inference}
In this section we address the problem of estimation of $\gq$ in \eqref{eq:WFSDE} from a continuously observed sample path $X^T = (X_t:\: t\in[0,T])$ in light of \thmref{thm:CU2022}. We first recall some standard theory (e.g.\ \citet[Ch.~9]{bas:pra:1980}, \citet[Ch.~6]{klo:etal:2003}) in the context of a one-dimensional SDE \eqref{eq:SDE}. For a model $\{\bPT{\gq}{T}:\: \gq \in \Theta\}$ for which the Radon--Nikodym derivative is available, the likelihood for $\gq$ is given by
\[
L_T(\gq) = \frac{\dd \bPT{\gq}{T}}{\dd \bPT{\gq_0}{T}},
\]
where $\bPT{\gq_0}{T}$ is a fixed dominating measure (here taken to be within the model). It is convenient to write the drift as
\[
\mu_\gq(X_t) = c(X_t) + a_\gq(X_t)
\]
with $a_{\gq_0}(\cdot) = 0$ so that $c(x)$ is the drift associated with $\bPT{\gq_0}{T}$.

If it holds that $\bPT{\gq}{T} \otimes \dd t$-a.e.\ $\gs(X_t) > 0$,
and if
\begin{equation}
\label{eq:information-explode}
\bPT{\gq}{T}((I_T)_{ij} < \infty) = 1, \qquad i,j=1,\dots,\dim\Theta,
\end{equation}
for all $\gq\in\Theta$, where $I_T$ is the ($\dim\Theta \times \dim\Theta$) observed information matrix
\[
I_T = \int_0^T \frac{\nabla_\gq a_\gq(X_t)[\nabla_\gq a_\gq(X_t)]^\top}{\gs^2(X_t)} \dd t,
\]
then $L_T(\gq)$ takes the form
\begin{equation}
\label{eq:likelihood}
L_T(\gq) = \exp\left(\int_0^T \frac{a_\gq(X_t)}{\gs^2(X_t)} \dd \widetilde{X}_t - \frac{1}{2}\int_0^T \frac{a_\gq^2(X_t)}{\gs^2(X_t)} \dd t\right),
\end{equation}
where
\[
\widetilde{X}_t \coloneq X_t - \int_0^t c(X_s) \dd s.
\]
It is possible to say more when the drift is linear in the parameters, $a_\gq(X_t) = Z(X_t)(\gq-\gq_0)$ with $Z$ ($= [\nabla_\gq a_\gq]^\top$) a ($1\times \dim\Theta$)-vector not depending on $\gq$ (and without loss of generality we may also assume the columns of $Z$ to be linearly independent so that $\Theta$ has minimal dimension). Then the MLE for $L_T(\gq)$ can be found in closed form. In particular, \eqref{eq:likelihood} simplifies to
\[
L_T(\gq) = \exp\left((\gq-\gq_0)^\top Y_T - \frac{1}{2}(\gq-\gq_0)^\top I_T (\gq-\gq_0)\right), \qquad Y_T \coloneq \int_0^T \frac{Z^\top(X_t)}{\gs^2(X_t)} \dd \widetilde{X}_t,
\]
whose maximum attained on $\bbR^{\dim \Theta}$ is
\begin{equation}\label{eq:MLE}
\widehat{\gq}_T = \gq_0 + I_T^{-1}Y_T.
\end{equation}
We recognise condition \eqref{eq:information-explode} as being equivalent to condition \eqref{eq:explode} holding for all $\gq_1\in\Theta$, since in the linear setting one can verify that $\int_0^T b^2(X_t) \dd t = (\gq-\gq_0)^\top I_T(\gq-\gq_0)$. We have seen that this guarantees the existence of $L_T(\gq)$ for each $T$ and, as might have been expected from the discussion above, we shall have to relax this condition to allow for $L_T(\gq)$ to exist only until a separating time.

\subsection{Wright--Fisher diffusion}
For the rest of this Section we assume $x_0 \in (0,1)$ so that any separating time is a.s.\ positive. We apply the reasoning above to the Wright--Fisher diffusion \eqref{eq:WFSDE} whose drift is, conveniently, linear in its parameters. We find that
\begin{align}
I_T &= \frac{1}{4}\begin{pmatrix}
\int_0^T\dfrac{1-X_t}{X_t} \dd t &~& -T &~& \int_0^T(1-X_t)\eta(X_t)\dd t\\[11pt]
-T &~& \int_0^T\dfrac{X_t}{1-X_t}\dd t &~& -\int_0^TX_t\eta(X_t)\dd t\\[11pt]
\int_0^T(1-X_t)\eta(X_t)\dd t &~& -\int_0^TX_t\eta(X_t)\dd t &~& \int_0^TX_t(1-X_t)\eta^2(X_t)\dd t
\end{pmatrix},\notag\\
&\eqcolon \frac{1}{4}\begin{pmatrix}
A_T & -T & C_T\\
-T & B_T & -D_T\\
C_T & -D_T & E_T
\end{pmatrix},\label{eq:I_T-WF}
\end{align}
and
\begin{equation}
\label{eq:Y_T}
Y_T = \frac{1}{2}\begin{pmatrix} \int_0^T\frac{\dd \widetilde{X}_t}{X_t}, & -\int_0^T \frac{\dd\widetilde{X}_t}{1-X_t}, & \int_0^T \eta(X_t) \dd \widetilde{X}_t\end{pmatrix}^\top,
\end{equation}
which can be substituted into \eqref{eq:MLE} to yield $\widehat{\gq}_T$.

The three stochastic integrals appearing in $Y_T$ can be eliminated by applying It\^o's lemma respectively to the functions $\log X_t$, $\log(1-X_t)$, and $H(X_t)$ [recall \eqref{eq:scale}]. We obtain
\begin{equation}\label{eq:Y_T-WF}
Y_T = \frac{1}{2}\begin{pmatrix} 
\log X_T - \log X_0 + \frac{1}{2}\int_0^T \frac{1-X_t}{X_t} \dd t - \int_0^T \frac{c(X_t)}{X_t} \dd t\\[11pt]
\log (1-X_T) - \log (1-X_0) + \frac{1}{2}\int_0^T \frac{X_t}{1-X_t} \dd t + \int_0^T \frac{c(X_t)}{1-X_t} \dd t\\[11pt]
H(X_T) - H(X_0) - \frac{1}{2}\int_0^T X_t(1-X_t)\eta'(X_t) \dd t - \int_0^T \eta(X_t)c(X_t) \dd t
\end{pmatrix},
\end{equation}
where we recall that $c(X_t) = \frac{1}{2}\left[\tup_0(1-X_t) - \tdown_0 X_t + \s_0X_t(1-X_t)\eta(X_t)\right]$. 

However $\widehat{\gq}_T$ is written, we have to invert a $3 \times 3$ matrix which is notationally troublesome and so we focus on some special cases. To avoid notational confusion we will denote $\widehat{\tup}^{(1)}_T$ the estimator for $\tup$ with the other two parameters assumed known (similarly for $\widehat{\tdown}^{(1)}_T$ and $\widehat{\s}^{(1)}_T$); $\widehat{\tup}^{(2)}_T$ the estimator for $\tup$ with $\s$ assumed known (similarly for $\widehat{\tdown}^{(2)}_T$); and $\widehat{\tup}_T$, $\widehat{\tdown}_T$, $\widehat{\s}_T$ the three components of the joint estimation problem in which none of the three parameters is known. 

First we derive the three marginal estimators for each parameter in turn, assuming the other two to be known. Through calculations similar to those above we find:
\begin{align}
\widehat{\tup}^{(1)}_T &= 1 + \dfrac{2\log X_T - 2\log X_0 + \tdown_0T - \s_0C_T}{A_T},
\label{eq:alpha-tilde}\\
\widehat{\tdown}^{(1)}_T &= 1 + \dfrac{2\log (1-X_T) - 2\log (1-X_0) + \tup_0T + \s_0D_T}{B_T},
\label{eq:beta-tilde}\\
\widehat{\s}^{(1)}_T &= \dfrac{2\ds\int_{X_0}^{X_T} \eta(y) \dd y - \int_0^T X_t(1-X_t)\eta'(X_t) + \eta(X_t)[\tup_0(1-X_t) - \tdown_0 X_t] \dd t}{E_T}. 
\label{eq:s-tilde}
\end{align}
In the special case of genic selection, $\eta(\cdot) = 1$, the latter estimator simplifies to
\[
\widehat{\s}^{(1)}_T = \dfrac{2(X_T - X_0) - \int_0^T [\tup_0(1-X_t) - \tdown_0 X_t)] \dd t}{\ds\int_0^T X_t(1-X_t) \dd t},
\]
which is studied in detail by \citet{wat:1979} and \citet{san:etal:2022}\footnote{There is a typo in the expression for $\widehat{\s}^{(1)}_T$ on p.~1751 of \citet{san:etal:2022}, where the factor of 2 in the numerator is missing. The parameter values quoted in the simulations should all be halved.}. The more general form for $\widehat{\s}^{(1)}_T$ in \eqref{eq:s-tilde}, and the estimators \eqref{eq:alpha-tilde}--\eqref{eq:beta-tilde} for $\tup$ and $\tdown$, appear to be new.

For the remainder of this section we focus on the problem of joint estimation of $(\tup,\tdown)$, with $\s$ assumed to be known but possibly nonzero. Then $I_T$ is the top-left $2\times 2$ submatrix in \eqref{eq:I_T-WF}, $Y_T$ is the vector \eqref{eq:Y_T-WF} restricted to its first two entries, and from \eqref{eq:MLE} the joint MLE for $(\tup, \tdown)$ is, after some simplification,
\begin{align} 
\widehat{\gq}^{(2)}_T \coloneq {}& \begin{pmatrix}
\widehat{\tup}^{(2)}_T\\
\widehat{\tdown}^{(2)}_T
\end{pmatrix} \notag\\
= {}& \frac{2}{A_TB_T - T^2} \notag\\
& {}\times
\begin{pmatrix}
B_T\log\left(\frac{X_T}{X_0}\right) + T\log\left(\frac{1-X_T}{1-X_0}\right) + \frac{B_T}{2}(A_T + T) + \frac{\s}{2}(D_TT-B_TC_T)\\[11pt]
A_T\log\left(\frac{1-X_T}{1-X_0}\right) + T\log\left(\frac{X_T}{X_0}\right) + \frac{A_T}{2}(B_T + T) + \frac{\s}{2}(A_TD_T - C_TT)
\end{pmatrix}.\label{eq:MLE-WF}
\end{align}

Two corrections to $\widehat{\gq}^{(2)}_T$ may be needed. First, the expression maximises $L_T(\gq)$ on $\bbR^2$ rather than on $\Theta = [0,\infty)^2$. Thus to ensure $\widehat{\gq}^{(2)}_T \in \Theta$ one should replace $\widehat{\tup}^{(2)}_T$ with $\widehat{\tup}^{(2)}_T \vee 0$ and $\widehat{\tdown}^{(2)}_T$ with $\widehat{\tdown}^{(2)}_T\vee 0$. However, it is simpler to study the asymptotic behaviour of the uncorrected estimator, as is done below.

Second, for certain models $L_t(\gq)$ is well defined only until a separating time $S$. Clearly it is possible for $S$ to be finite, manifest in $(I_t)_{11}$ exploding at time $\inf\{t \geq 0:\: A_t = \infty\}$ or $(I_t)_{22}$ exploding at time $\inf\{t \geq 0:\: B_t = \infty\}$. Recalling the results of \sref{sec:separating}, we recognise these as the respective times $T_0$ and $T_1$ which are both mutually separating for $(\bP{\gq},\bP{\gq_0})$ (unless $\gq = \gq_0$). We should therefore define a new estimator to account for the possibility of separation.

To investigate, let us walk through an example; other cases will follow by similar arguments. Suppose that, for some $\omega \in \Omega$, we have $0 < T_0(\omega) \leq T < T_1(\omega)$. Letting $t\uparrow T_0(\omega)$, we find $A_t(\omega) \to \infty$ and hence from \eqref{eq:MLE-WF} that
\begin{equation}\label{eq:crystallize}
\begin{pmatrix}
\widehat{\tup}^{(2)}_t(\omega)\\[11pt]
\widehat{\tdown}^{(2)}_t(\omega)
\end{pmatrix}
\to
\begin{pmatrix} 
1 + 2\lim_{t\uparrow T_0(\omega)}\frac{\log X_t(\omega)}{A_t(\omega)}\\[11pt]
1 + \frac{2\log\frac{1-X_{T_0}(\omega)}{1-X_0} + \left[1+2\lim_{t\uparrow T_0(\omega)} \frac{\log X_t(\omega)}{A_t(\omega)}\right]T_0(\omega) + \s D_{T_0}(\omega)}{B_{T_0}(\omega)}
\end{pmatrix}.
\end{equation}
The estimate $\widehat{\tup}^{(2)}_t(\omega)$ `crystallizes' according to the limit of $\log(X_t(\omega))/A_t(\omega)$ as $t \uparrow T_0(\omega)$. As is suggested from our intuition about what it means for ``observed information'' to explode, we have in fact learned $\tup$ without error: $\lim_{t\uparrow T_0} \widehat{\tup}^{(2)}_t = \tup$, $\bP{\gq}$-a.s. (This claim will be verified below.) No additional information from the trajectory beyond time $T_0$ is needed in the estimation of $\tup$. Now we recognise what has happened to $\widehat{\tdown}^{(2)}_t(\omega)$: the limit in \eqref{eq:crystallize} is none other than the univariate estimate $\widehat{\tdown}^{(1)}_{T_0}(\omega)$ from \eqref{eq:beta-tilde} except with $\tup_0$ replaced by its newly learned value, $1+2\lim_{t\uparrow T_0(\omega)} \log(X_t(\omega))/A_t(\omega)$ ($= \tup$). From time $T_0$ onwards the joint estimation problem automatically reduces to a single-parameter estimation problem. Continuing beyond $T_0$ up to time $T$ we estimate $\tdown$ in the usual way. Pleasingly, all of this behaviour is completely encapsulated in \eqref{eq:MLE-WF}.

To complete the argument, we ought to check that the remaining contributions in \eqref{eq:crystallize}, from $\log(1-X_t(\omega))$, $D_{t}(\omega)$, and $B_{t}(\omega)^{-1}$, are nonexplosive as $t\uparrow T_0(\omega)$. To see this, note that with $0 < T_0(\omega) \leq T < T_1(\omega)$ by sample path continuity there exists some $\varepsilon_1(\omega) \in (x_0,1)$ such that $X_t(\omega) \leq \varepsilon_1(\omega)$ for all $t \in [0,T_0(\omega)]$, and there exists some $\varepsilon_2(\omega) \in (0,T_0(\omega))$ such that $X_t(\omega) \geq x_0/2 > 0$ for all $t \in [0,\varepsilon_2(\omega)]$. From this we can bound $\log(1-X_t(\omega))$, $D_{t}(\omega)$, and $B_{t}(\omega)$, respectively [recalling \eqref{eq:I_T-WF}] by
\begin{align*}
\log(1-\varepsilon_1(\omega)) &\leq \log(1-X_{t}(\omega)) \leq 0,\\
-\varepsilon_1(\omega)t\sup_{x\in[0,1]}|\eta(x)| &\leq D_{t}(\omega) \leq \varepsilon_1(\omega)t\sup_{x\in[0,1]}|\eta(x)|,\\
0 < \varepsilon_2(\omega)\frac{x_0}{2} &\leq B_{t}(\omega) \leq \frac{\varepsilon_1(\omega)}{1-\varepsilon_1(\omega)}t,
\end{align*}
verifying that $\log(1-X_t(\omega))$ and $D_{t}(\omega)$, and $B_{t}(\omega)^{-1}$ remain finite as $t\uparrow T_0(\omega)$.

The above arguments, extended to other cases, suggests that the appropriate correction to $\widehat{\gq}^{(2)}_T$ is given by the estimator
 \begin{align}
\widehat{\vartheta}_T = {}&
\bbI_{\{T < S\}}\begin{pmatrix}
\widehat{\tup}^{(2)}_T\\[11pt]
\widehat{\tdown}^{(2)}_T
\end{pmatrix}\notag\\
 & {}+ \bbI_{\{T_0 < T\wedge T_1\}}\begin{pmatrix}
1+2\lim_{t\uparrow T_0} \frac{\log X_t}{A_t}\\[11pt]
\lim_{t\uparrow (T\wedge T_1)} \widehat{\tdown}^{(1)}_{t}(\widehat{\tup}^{(2)}_{T_0-})
\end{pmatrix}
 + \bbI_{\{T_1 < T\wedge T_0\}}\begin{pmatrix}
\lim_{t\uparrow (T\wedge T_0)}\widehat{\tup}^{(1)}_{t}(\widehat{\tdown}^{(2)}_{T_1-})\\[11pt]
1+2\lim_{t\uparrow T_1} \frac{\log(1- X_t)}{B_t}
\end{pmatrix}, \label{eq:vartheta}
\end{align}
where $\widehat{\tup}^{(1)}_t(\tdown_0)$ is the estimator \eqref{eq:alpha-tilde} with its dependence on $\tdown_0$ given explicitly, so $\widehat{\tup}^{(1)}_{t}(\widehat{\tdown}^{(2)}_{T_1-})$ is the same estimator with $\tdown_0$ replaced by $\lim_{u\uparrow T_1} \widehat{\tdown}^{(2)}_u$, as discussed above; $\widehat{\tdown}^{(1)}_{t}(\widehat{\tup}^{(2)}_{T_0-})$ is defined similarly.

To summarise, if the process hits a separating endpoint then the manner in which the boundary is hit `reveals' the corresponding mutation parameter and we infer it without error. Otherwise we cannot distinguish the two parameters with certainty in finite time. Similar phenomena are known to arise in some other models: (i) when estimating the immigration parameter associated with the CBI diffusion on the approach to extinction \citep{ove:1998}, and (ii) when estimating the recombination parameter in a multi-locus Wright--Fisher diffusion model on the approach to extinction of a haplotype \citep{gri:jen:2023}. In both of these examples, estimators are constructed from a continuously observed diffusion using the general arguments outlined at the beginning of this Section, and corrections to the na\"ive estimator \eqref{eq:MLE}, akin to \eqref{eq:vartheta}, are proposed. That the immigration parameter of the CBI diffusion can be learned without error is perhaps not surprising in light of our earlier observation that $X$ behaves asymptotically near $0$ as a squared Bessel process, since the squared Bessel process \eqref{eq:BESQ} is a special case of the CBI diffusion with critical branching and immigration parameter $\kappa$.

The following result justifies our claim that $\widehat{\vartheta}_T$ exhibits the right asymptotic behaviour. 
\begin{theorem}~\label{thm:consistency} Suppose $x_0 \in (0,1)$. 
\begin{enumerate}[(i)]
\item If $\tup, \tdown \geq 1$ then $\widehat{\gq}^{(2)}_T$ in \eqref{eq:MLE-WF} is strongly consistent.
\item If $\tup, \tdown > 0$ then $\widehat{\vartheta}_T$ in \eqref{eq:vartheta} is strongly consistent.
\item If $\tup = 0$ and $\tdown > 0$ then the first component of $\widehat{\vartheta}_T$ is a strongly consistent estimator for $\tup$; there is no consistent estimator for $\tdown$.
\item If $\tup > 0$ and $\tdown = 0$ then the second component of $\widehat{\vartheta}_T$ is a strongly consistent estimator for $\tdown$; there is no consistent estimator for $\tup$.
\item If $\tup = 0 = \tdown$ then there is neither a consistent estimator for $\tup$ nor for $\tdown$. 
\end{enumerate}
\end{theorem}

\begin{proof}
(i) It is convenient to write $Y_T$ in \eqref{eq:Y_T} in the form
\begin{equation}
\label{eq:Y_T2}
Y_T = \frac{1}{2}\cM_T + \frac{1}{2}\begin{pmatrix} \int_0^T\frac{a_\gq(X_t)}{X_t} \dd t, & -\int_0^T \frac{a_\gq(X_t)}{1-X_t} \dd t, & \int_0^T \eta(X_t)a_\gq(X_t) \dd t\end{pmatrix}^\top,
\end{equation}
where $\cM_T$ is the vector
\begin{multline}
\cM_T = (M_T, N_T, Q_T)^\top\\
= \begin{pmatrix}
\int_0^T \sqrt{\frac{1-X_t}{X_t}} \dd W_t, & -\int_0^T \sqrt{\frac{X_t}{1-X_t}}\dd W_t, & \int_0^T \eta(X_t)\sqrt{X_t(1-X_t)} \dd W_t\end{pmatrix}^\top\label{eq:M}
\end{multline}
whose quadratic variation process is $\langle \cM \rangle_T = 4I_T$ with $I_T$ as in \eqref{eq:I_T-WF}. Then in fact
\begin{align}
\label{eq:Y_T3}
Y_T &= \frac{1}{2}\cM_T + I_T(\gq-\gq_0).
\end{align}
We first observe that $Y_T$ is well defined: expression \eqref{eq:Y_T3} shows that $Y_T$ is well defined when $\cM_T$ and $I_T$ are. In particular, the Lebesgue integrals in \eqref{eq:Y_T2} are finite if those appearing in $I_T$ in \eqref{eq:I_T-WF} are finite, which holds at least until a separating time using arguments similar to those given above. For example, when $\tup, \tdown \geq 1$, for any $T \in [0,\infty)$ one can appeal to sample path continuity to find $\varepsilon_1$ such that $\bP{\gq}(X_t \leq \varepsilon_1 < 1 \text{ for all }t \in [0,T]) = 1$, and hence $B_t \leq \int_0^t \varepsilon_1/(1-\varepsilon_1) \dd s < \infty$ $\bP{\gq}$-a.s. The remaining entries of $I_T$ can be shown to be finite in a similar manner. Furthermore, $\cM_T$ is well defined if the integrands appearing in \eqref{eq:M} are square-integrable with respect to time \citep[e.g.][Def.~3.2.22 and Prop.~3.2.24, p.~146--147]{kar:shr:1998}, and this is verified by exactly the same integrability conditions we have just checked, namely \eqref{eq:information-explode}.

Next, by comparing \eqref{eq:Y_T3} with \eqref{eq:MLE} we obtain an expression for the estimator error:
\begin{align} \label{eq:error}
\widehat{\gq}_T - \gq &= 2\langle \cM \rangle_T^{-1} \cM_T.
\end{align}
In the $2\times 2$ subproblem for $(\tup,\tdown)^\top$ this can be written explicitly (after some rearrangement) in the form 
\begin{equation}\label{eq:error-WF}
\begin{pmatrix}
\widehat{\tup}^{(2)}_T\\
\widehat{\tdown}^{(2)}_T
\end{pmatrix}
= \begin{pmatrix}
{\tup}\\
{\tdown}
\end{pmatrix}
+\frac{2}{1-\dfrac{T}{\langle M\rangle_T}\dfrac{T}{\langle N \rangle_T}}\begin{pmatrix}
\frac{M_T}{\langle M \rangle_T} + \frac{N_T}{\langle N\rangle_T}\frac{T}{\langle M\rangle_T}\\[11pt]
\frac{N_T}{\langle N \rangle_T} + \frac{M_T}{\langle M\rangle_T}\frac{T}{\langle N\rangle_T}
\end{pmatrix}
\end{equation}
(and $\langle M\rangle_T = A_T$, $\langle N\rangle_T = B_T$). Now suppose $\tup,\tdown > 1$. Then $X$ is ergodic with
\begin{align} \label{eq:ergodic}
\frac{\langle M \rangle_T}{T} &\to \bbE_\gq\left(\frac{1-X_\infty}{X_\infty}\right) \eqcolon a_\infty, &
\frac{\langle N \rangle_T}{T} &\to \bbE_\gq\left(\frac{X_\infty}{1-X_\infty}\right) \eqcolon b_\infty,
\end{align}
$\bP{\gq}$-a.s.\ as $T\to\infty$. It is easily checked, by reading the stationary measure for $X_\infty$ from \eqref{eq:speed1}, that $a_\infty,b_\infty < \infty$. Moreover, each entry of $\cM_T$ is a $\bP{\gq}$-martingale and so by the strong law of large numbers for local martingales \citep[Ch.~V, Exercise~1.16, p.~186]{rev:yor:1999} we know that, $\bP{\gq}$-a.s.\ as $T\to\infty$,
\begin{align} \label{eq:SLLNM}
\frac{M_T}{\langle M \rangle_T} &\to 0 & \text{and} &&
\frac{N_T}{\langle N \rangle_T} &\to 0.
\end{align}
Thus, combining \eqref{eq:ergodic} and \eqref{eq:SLLNM} with \eqref{eq:error-WF}, we find
\[
\begin{pmatrix}
\widehat{\tup}^{(2)}_T\\
\widehat{\tdown}^{(2)}_T
\end{pmatrix}
\overset{\bP{\gq}\text{-a.s.}}{\longrightarrow} \begin{pmatrix}
{\tup}\\
{\tdown}
\end{pmatrix}, \qquad T\to\infty,
\]
as required.

Next suppose $\tup = 1$, $\tdown > 1$. We still have that $X$ is ergodic but now with $a_\infty = \infty$. This does not preclude an appeal to the strong law of large numbers for a positive integrand, and (after multiplying throughout by $\langle M \rangle_T\langle N\rangle_T/T^2$) one can check that the error in \eqref{eq:error-WF} still goes to $0$.

The cases $(\tup,\tdown) \in (1,\infty)\times \{1\}$ and $(\tup,\tdown) = (1,1)$ can be argued in a similar way.

Beyond case (i), it is no longer the case that $I_T$ remains finite up to any finite time $T$, and one must investigate the way in which $\cM_T$ and $I_T$ explode on the route to a separating time.

(ii) Suppose first that $\tup,\tdown \in (0,1)$. The process $M_T$ is a continuous local martingale on the stochastic interval $[0,T_0)$ (see \citet[Ch.~IV, Exercise~1.48, p.~136]{rev:yor:1999} for a definition and \citet[p.~894]{mij:uru:2015} for further discussion and references) with
\[
\lim_{T\uparrow T_0} \langle M \rangle_{T} = \infty,
\]
and localising sequence $(\inf\{t\geq 0:\: |M_t| \geq n\} :\: n =0,1,\dots)$. As noted by \citet[Ch.~V, Exercise~1.18, p.~187]{rev:yor:1999}, the Dambis--Dubins--Schwarz theorem can be modified to apply to a continuous local martingale on a stochastic interval. Hence the strong law of large numbers for local martingales can be similarly modified, so we have that, $\bP{\gq}$-a.s.,
\begin{align*}
\frac{M_{T}}{\langle M \rangle_{T}} &\to 0, & T\uparrow T_0,
\end{align*}
and similarly
\begin{align*}
\frac{N_{T}}{\langle N \rangle_{T}} &\to 0, & T\uparrow T_1.
\end{align*}
Suppose for simplicity we are on the event $\{T_0 < T_1\}$. The relevant term in the estimator \eqref{eq:vartheta} is constructed first by letting $T \uparrow T_0$. Taking this limit in \eqref{eq:error-WF} and using that $T_0 < \infty$ $\bP{\gq}$-a.s., we find
 \begin{equation}\label{eq:stopped1}
\lim_{T\uparrow T_0}\begin{pmatrix}
\widehat{\tup}^{(2)}_T\\
\widehat{\tdown}^{(2)}_T
\end{pmatrix}
= \begin{pmatrix}
{\tup}\\
{\tdown}
\end{pmatrix}
+\begin{pmatrix}
0\\[11pt]
\frac{2N_{T_0}}{\langle N \rangle_{T_0}}
\end{pmatrix},
\end{equation}
showing that $\widehat{\tup}^{(2)}_T$ is strongly consistent and that the error for $\widehat{\tdown}^{(2)}_{T_0}$ is of the same form as the one-parameter problem for $\widehat{\tdown}^{(1)}_T$ at time $T_0$. Continuing, we next let $T\uparrow T_1$ so that the second component of \eqref{eq:stopped1} becomes 
$\lim_{T\uparrow T_1} \widehat{\tdown}^{(2)}_T = \tdown$, 
as required. The case $\{T_1 < T_0\}$ can be handled in a similar fashion.

If $\tup < 1$ and $\tdown \geq 1$ (or vice versa) then strong consistency is achieved by combining the arguments from part (i) for $N_T$ and part (ii) for $M_T$; we omit the details. If both $\tup \geq 1$ and $\tdown \geq 1$ then $\widehat{\vartheta}_T = \widehat{\gq}^{(2)}_T$ which is covered by part (i).

(iii) Here we know that $T_0 < \infty$ a.s., and strong consistency of $\widehat{\vartheta}_T$ for $\tup$ follows as in part (ii). Now suppose there exists some consistent estimator $\widetilde{\beta}_T$ for $\tdown$. Then for each $\varepsilon > 0$:
\begin{align*}
\lim_{T\to\infty}\bP{0,\tdown,\s }\left(|\widetilde{\beta}_T - \tdown| > \varepsilon\right) &= 0, & \lim_{T\to\infty}\bP{0,\tdown_1,\s }\left(|\widetilde{\beta}_T - \tdown_1| > \varepsilon\right) &= 0,
\end{align*}
for $\tdown_1 \neq \tdown$, which implies $\bP{0,\tdown,\s } \perp \bP{0,\tdown_1,\s }$ \citep[p.~1766]{wei:win:1990}. But this contradicts \thmref{thm:CU2022}, where it is shown that $S = \overline{T}_1$ (see rows ({\sc vi}) and ({\sc x}) of \tref{tab:S}) $\bP{0,\tdown,\s }, \bP{0,\tdown_1,\s }$-a.s. Put another way, there is a positive probability that $T_1 = \infty$ in which case $S=\gd$, $\langle N \rangle_\infty < \infty$, and the error in $\widetilde{\beta}_\infty$ remains nonzero.

(iv) This follows as in (iii).

(v) From \eqref{eq:a} we know $\bP{0,0,\s }(T_0 \wedge T_1 < \infty) = 1$. Since both endpoints are absorbing, exactly one of $\{T_0 < \infty = T_1\}$ or $\{T_1 < \infty = T_0\}$ occurs. Furthermore, each of these events has positive probability (specifically, $[S_\gq(1)-S_\gq(x_0)]/[S_\gq(1)-S_\gq(0)]$ and $[S_\gq(x_0)-S_\gq(0)]/[S_\gq(1)-S_\gq(0)]$ respectively; see \citet[Thm.~33.7, p.~757]{kal:2021}, though the exact formulas are not required here). On the first event any estimator for $\tdown$ cannot be consistent by the same argument as in (iii); similarly, on the second event any estimator for $\tup$ cannot be consistent.
\end{proof}

\begin{corollary}~\label{cor:bisecting} 
Suppose $x_0 \in (0,1)$.
\begin{enumerate}[(i)]
\item If $\tup < 1$ and $\tdown > 0$ then
\[
\bP{\gq}\left(\lim_{t\uparrow T_0}\frac{\log X_t}{\int_0^t \frac{1-X_u}{X_u}\dd u} = \frac{\tup-1}{2}\right) = 1.
\]
\item If $\tdown < 1$ and $\tup > 0$ then
\[
\bP{\gq}\left(\lim_{t\uparrow T_1}\frac{\log (1-X_t)}{\int_0^t \frac{X_u}{1-X_u}\dd u} = \frac{\tdown-1}{2}\right) = 1.
\]
\end{enumerate}
\end{corollary}
\begin{proof}
This is a restatement of \thmref{thm:consistency}(ii) in the respective cases $T_0 < \infty$ and $T_1 < \infty$ $\bP{\gq}$-a.s., taking note of the explicit limit of $\widehat{\vartheta}_T$ on the approach to $T_0$ or $T_1$, as given in \eqref{eq:vartheta}.
\end{proof}
Corollary \ref{cor:bisecting} provides another verification that $(\bP{\gq},\bP{\gq_0})$ will separate on hitting a reflecting endpoint if their corresponding mutation parameters differ, and further shows us how to extract the mutation parameter from the sample path.

For certain submodels we can also obtain a central limit theorem. Here we can return to the joint problem for estimating all three of $\tup$, $\tdown$, and $\s$.
\begin{theorem}\label{thm:CLT}
Suppose $x_0 \in (0,1)$, and let $\widehat{\gq}_T$ as in \eqref{eq:MLE} be an estimator for $\gq = (\tup,\tdown,\s )^\top$. If $\tup, \tdown > 1$ then
\begin{equation}\label{eq:CLT}
\sL_\gq\left(\sqrt{T}(\widehat{\gq}_T - \gq)\right) \to \cN(0,4\Sigma^{-1});
\end{equation}
that is, $\sqrt{T}$ times the estimator error converges in distribution to a normal distribution with mean 0 and covariance matrix $4\Sigma^{-1}$, where
\[
\Sigma \coloneq \begin{pmatrix}
a_\infty & -1 & c_\infty\\
-1 & b_\infty & -d_\infty\\
c_\infty & -d_\infty & e_\infty
\end{pmatrix},
\]
with $a_\infty$ and $b_\infty$ given by \eqref{eq:ergodic} and
\begin{align*}
c_\infty &\coloneq \bbE_\gq\left((1-X_\infty)\eta(X_\infty)\right), &
d_\infty &\coloneq \bbE_\gq(X_\infty\eta(X_\infty)),\\
e_\infty &\coloneq \bbE_\gq(X_\infty(1-X_\infty)\eta^2(X_\infty)).
\end{align*}
\end{theorem}

\begin{proof}
With $\cM_T$ as in \eqref{eq:M}, we know as in the proof of \thmref{thm:consistency} that $\langle \cM \rangle_T/T \overset{\text{a.s.}}{\longrightarrow} \Sigma$ as $T\to\infty$. Since $\tup,\tdown > 1$, the entries of $\Sigma$ are all finite. Hence by the martingale central limit theorem (see \citet[Thm.~1.19, p.~43]{kut:2004} for a statement in the univariate case; a multidimensional version can be constructed from the Cram\'er--Wold theorem and the fact that all of the projections of a multivariate normal distribution are again normal):
\[
\sL_\gq\left(\frac{1}{\sqrt{T}}\cM_T\right) \to \cN(0,\Sigma).
\]
Equation \eqref{eq:CLT} then follows from \eqref{eq:error}.
\end{proof}
\begin{remark}
Integrals of the form appearing in $\Sigma$ are well known in the study of non-neutral population genetics models. They are intractable in general. If the model is neutral ($\s=0$) and we take $\eta(\cdot) = 1$ then $\Sigma$ can be given exactly:
\[
\Sigma = \begin{pmatrix}
\tdown/(\tup-1) & -1 & \tdown/(\tup+\tdown)\\[11pt]
-1  & \tup/(\tdown-1) & -\tup/(\tup+\tdown)\\[11pt]
 \tdown/(\tup+\tdown) & -\tup/(\tup+\tdown) & \tup\tdown/[(\tup+\tdown)(\tup+\tdown+1)]
\end{pmatrix}.
\]
\end{remark}
\begin{remark}
Notice that, unlike \thmref{thm:consistency}, the result in \thmref{thm:CLT} allows for neither $\tup = 1$ nor $\tdown = 1$. Although such models continue to exhibit inaccessible boundaries (and thus $\bPT{\gq}{T}((I_T)_{ij} < \infty) = 1$ for each $i,j$) they occupy an inconvenient borderline case for which some entries ($a_\infty$ or $b_\infty$) of $\Sigma$ are infinite. It seems intuitively possible that the mutation rate is strong enough to keep the boundaries inaccessible but weak enough that the trajectory might get sufficiently close to the boundary for the rate of learning of a parameter to be \emph{faster} than the usual $T^{-1/2}$. To obtain a more general central limit theorem allowing for $\tup = 1$ or $\tdown = 1$ would require a study of the asymptotic behaviour of $Y_T$ in \eqref{eq:Y_T-WF}. An analogous study for the CBI diffusion is conducted by \citet{ben:keb:2012,ben:keb:2013}.

A central limit theorem allowing for $\tup < 1$ or $\tdown < 1$ would further require the study of $Y_T$ along sequences of stopping times such as those of the form $T_n^{(i,j)} = \inf\{t\geq 0:\: (I_t)_{ij} = n\}$, $n=1,2,\dots$.
\end{remark}

\paragraph*{Data access statement.} No new data was generated or analysed in this study. 

\paragraph*{Acknowledgements.} I would very much like to thank Aleks Mijatovi\'c and Dario Span\`o for helpful conversations, as well as an anonymous referee who provided several constructive suggestions including simplifications of the proofs of Theorems \ref{thm:CU2022} and  \ref{thm:reflecting-f}.


\paragraph*{Funding.} This work was supported by The Alan Turing Institute under the Engineering and Physical Sciences Research Council (EPSRC) grant EP/N510129/1 and via ProbAI: A Hub for the Mathematical and Computational Foundations of Probabilistic AI (EP/Y028783/1).

\appendix
\renewcommand{\thesection}{\Alph{section}}
\renewcommand{\thedefinition}{\Alph{section}.\arabic{definition}}

\section*{Appendix}

\section{Separating points of general diffusions}
\label{sec:separating-points}
In this section we give the definition of a non-separating point for a general diffusion on $I$ characterised by its scale function and speed measure.
\begin{definition}[\citet{cri:uru:2026}]
\label{def:good-interior}
A point $x \in I^\circ$ is a \emph{non-separating interior point} if there exists an open neighbourhood $N(x) \subset I^\circ$ of $x$ such that
\begin{enumerate}[(i)]
\item The right derivative
\[
\frac{\dd^+ S_{\gq_0}}{\dd S_{\gq_1}}(y) \coloneq \lim_{h \downarrow 0} \frac{S_{\gq_0}(y+h) - S_{\gq_0}(y)}{S_{\gq_1}(y+h) - S_{\gq_1}(y)}
\]
exists for all $y\in N(x)$ and is positive and finite;
\item there exists a Borel function $\varrho:N(x)\to\bbR$ such that
\begin{align}
\int_{N(x)} \varrho^2(y) S_{\gq_1}(y) \dd y &< \infty,\label{eq:good1}\\
\frac{\dd^+ S_{\gq_0}}{\dd S_{\gq_1}}(y) - \frac{\dd^+ S_{\gq_0}}{\dd S_{\gq_1}}(x) &= \int_x^y \varrho(z) S_{\gq_0}(z) \dd z, & y &\in N(x); \label{eq:good2}
\end{align}
\item the derivative
\[
\frac{\dd m_{\gq_0}}{\dd m_{\gq_1}}(y) \coloneq \lim_{h \downarrow 0} \frac{m_{\gq_0}((y-h,y+h))}{m_{\gq_1}((y-h,y+h))}
\]
exists for all $y\in N(x)$ and, on $N(x)$,
\begin{equation}
\label{eq:speed-times-scale}
\frac{\dd m_{\gq_0}}{\dd m_{\gq_1}}(y) \frac{\dd^+ S_{\gq_0}}{\dd S_{\gq_1}}(y) = 1.
\end{equation}
\end{enumerate}
\end{definition}

\begin{definition}[\citet{cri:uru:2026}]
\label{def:good-endpoint}
The point $x = \inf I$ is a \emph{non-separating left endpoint} if there exists a non-empty open interval $N(x) \subset I^\circ$ with $x$ as left endpoint such that
\begin{enumerate}[(i)]
\item $x$ is `half-good': (a) all points in $N(x)$ are non-separating interior points, (b) $S_{\gq_1}(x) \coloneq \lim_{y \downarrow x} S_{\gq_1}(y) \in \bbR$, and (c)
\begin{equation}
\label{eq:half-good}
\int_{N(x)} |S_{\gq_1}(y) - S_{\gq_1}(x)|\varrho^2(y) S_{\gq_1}(y) \dd y < \infty,
\end{equation}
with $\varrho$ as in \dref{def:good-interior};
\item if $x \in I_{\gq_1}$ (equivalently $x\in I_{\gq_0}$ if $x$ is half-good) then either $m_{\gq_0}(\{x\}) = m_{\gq_1}(\{x\}) = \infty$ or $m_{\gq_0}(\{x\}) < \infty$, $m_{\gq_1}(\{x\}) < \infty$;
\item if $x \in I_{\gq_1}$ (equivalently $x\in I_{\gq_0}$ if $x$ is half-good) and $m_{\gq_0}(\{x\}) < \infty$, $m_{\gq_1}(\{x\}) < \infty$, then:
\begin{enumerate}[(a)]
\item the right derivative $\frac{\dd^+ S_{\gq_0}}{\dd S_{\gq_1}}(x)$ exists and is positive and finite;
\item there exists a Borel function $\varrho:N(x)\to\bbR$ such that \eqref{eq:good1} and \eqref{eq:good2} hold;
\item the derivative
\[
\frac{\dd m_{\gq_0}}{\dd m_{\gq_1}}(x) \coloneq \lim_{y \downarrow x} \frac{m_{\gq_0}([x,y))}{m_{\gq_1}([x,y))}
\]
exists and is positive and finite, and \eqref{eq:speed-times-scale} holds at $x$.
\end{enumerate}
\end{enumerate}
\end{definition}
A similar definition can be given for a non-separating right endpoint.

\begin{proof}[Proof of \lemmaref{lem:separating-points}]
(a) Fix $x \in (0,1)$ and let $N(x) = (x-\varepsilon,x+\varepsilon)$ be such that $0 < x-\varepsilon < x < x+\varepsilon < 1$. First we check conditions (i--iii) of \dref{def:good-interior}. Using \eqref{eq:scale}--\eqref{eq:speed2}, we find
\begin{align*}
\frac{\dd^+ S_{\gq_0}}{\dd S_{\gq_1}}(y) &= C_{\gq_0}C_{\gq_1}^{-1}y^{-(\tup_0-\tup_1)}(1-y)^{-(\tdown_0-\tdown_1)}e^{-(\s_0-\s_1)H(y)},\\
\frac{\dd m_{\gq_0}}{\dd m_{\gq_1}}(y) &= C_{\gq_1}C^{-1}_{\gq_0}y^{\tup_0-\tup_1}(1-y)^{\tdown_0-\tdown_1}e^{(\s_0-\s_1)H(y)},
\end{align*}
so conditions (i) and (iii) hold on $N(x)$ and the only difficulty is in exhibiting a function $\varrho$ to satisfy (ii). We claim that (ii) is satisfied by the choice
\begin{equation}
\label{eq:beta}
\varrho(y) = C_{\gq_1}^{-1}y^{\tup_1}(1-y)^{\tdown_1}e^{\s_1H(y)}\left(\frac{\tdown_0-\tdown_1}{1-y} - \frac{\tup_0-\tup_1}{y} - (\s_0 - \s_1)\right),
\end{equation}
which is easily verified by a direct calculation. Thus all interior points $x\in (0,1)$ are non-separating.

(b) We check conditions (i--iii) of \dref{def:good-endpoint}. For the endpoint $0$, the appropriate neighbourhood is $N(0) = (0,\varepsilon) \subset (0,1)$. Checking (i) we find
\begin{align} \label{eq:S0}
S_{\gq_1}(0) &\coloneq \lim_{x\downarrow 0} S_{\gq_1}(x) \begin{cases}
\in \bbR & \tup_1 < 1,\\
= -\infty & \tup_1 \geq 1,
\end{cases}
\end{align}
so we need to determine whether \eqref{eq:half-good} holds when $\tup_1 < 1$. Note that if $\varrho$ is given as in \eqref{eq:beta} then $\varrho^2$ can be written in the form
\[
\varrho^2(y) = C_{\gq_1}^{-2}y^{2\tup_1-2}(1-y)^{2\tdown_1}e^{2\s_1H(y)}[(\tup_0-\tup_1)^2 + y\phi(y)]
\]
for some continuous, bounded function $\phi:[0,\epsilon]\to\bbR$ whose exact form will turn out to be unimportant, so that
\begin{multline}
\label{eq:lgoodWF}
|S_{\gq_1}(y) - S_{\gq_1}(0)|\varrho^2(y) S_{\gq_1}(y) =\\ \left|\int_0^y z^{-\tup_1}(1-z)^{-\tdown_1}e^{-\s_1H(z)} \dd z\right|[(\tup_0-\tup_1)^2 + y\phi(y)]y^{\tup_1-2}(1-y)^{\tdown_1}e^{\s_1H(y)},
\end{multline}
which is integrable on $N(0)$ if and only if
\[
G(x) \coloneq \left(\int_0^x z^{-\tup_1} \dd z\right)[(\tup_0-\tup_1)^2 + x\phi(x)]x^{\tup_1-2}
\]
is integrable on $N(0)$. Recalling our assumption that $\tup_1 < 1$, we have:
\[
G(x) = \frac{x^{1-\tup}}{1-\tup_1}[(\tup_0-\tup_1)^2 + x\phi(x)]x^{\tup_1-2},
\]
so, noting that $\phi$ is integrable on $N(0)$,
\begin{align*}
\int_{0+}^x G(y) \dd y &= \frac{1}{1-\tup_1}\int_{0+}^x [(\tup_0-\tup_1)^2y^{-1} + \phi(y)] \dd y\\
&= \lim_{z\downarrow 0} \frac{(\tup_0-\tup_1)^2}{1-\tup_1}(\log x - \log z) + \frac{1}{1-\tup_1}\int_0^x \phi(y) \dd y & \begin{cases}
< \infty & \tup_0 = \tup_1,\\
= \infty & \tup_0 \neq \tup_1.
\end{cases}
\end{align*}
Hence for $\tup_1 < 1$, the expression \eqref{eq:lgoodWF} is integrable on $N(0)$ if and only if $\tup_0 = \tup_1$. Combining this observation with \eqref{eq:S0} we conclude that the point $0$ satisfies condition (i) of \dref{def:good-endpoint} if and only if $\tup_0 = \tup_1 < 1$. Under this (highly restrictive!) condition, it is straightforward to verify that the remaining conditions (ii) and (iii) of \dref{def:good-endpoint} hold. In summary, $0$ is a non-separating endpoint if and only if $\tup_0 = \tup_1 < 1$.

(c) This is similar to case (b) and so is omitted.
\end{proof}

\begin{remark}
As discussed by \citet{des:etal:2021} and \citet{cri:uru:2026}, the `half-good' condition \eqref{eq:half-good} for an accessible endpoint pertains to non-separation on the \emph{approach} to that endpoint, while condition \eqref{eq:good1}--\eqref{eq:good2} pertains to non-separation on \emph{emergence} from the endpoint. The proof of \lemmaref{lem:separating-points} verifies that if $0$ is an accessible endpoint then separation occurs on the approach to $0$ unless $\tup_0 = \tup_1$. It is possible to show that the same condition, $\tup_0 = \tup_1$, is required to guarantee non-separation on emergence from $0$ (which also follows from \thmref{thm:all}).
\end{remark}

\section{Separating points of It\^o SDEs}
\label{sec:speed-and-scale}
In this section we state a result which provides some intuition for Definition \ref{def:non-separating}; that is, for the conditions under which an interior point $z \in I^\circ$ can be said to be \emph{non-separating}, or \emph{good}, with respect to a pair $(\bP{\gq_0},\bP{\gq_1})$ of path measures for the diffusion $X$ characterised by \eqref{eq:SDE} under $\bP{\gq}$. Specifically, we would like to compute a separating time in order to determine up to what time \eqref{eq:equivalent} holds. A key observation is that one can use the occupation times formula to convert a pathwise statement such as the event appearing in \eqref{eq:explode} into a \emph{deterministic} statement about the local integrability of a function. Similar ideas underlie the Engelbert--Schmidt zero-one law \citep[e.g.][Prop.~3.6.27, p.~216]{kar:shr:1998}.

As is usual, assume conditions on $\mu_\theta$ and $\sigma$ such that \eqref{eq:SDE} admits a weak solution, unique in law. Consider an interval $(c,d) \subseteq I$ such that $x_0 \in (c,d)$ and $z \in (c,d)$, and assume $\sigma(x) \neq 0$ for $x \in (c,d)$. Denote the first exit time from $(c,d)$ by $T_{c,d} = T_c \wedge T_d$, and let $f:I\to[0,\infty)$ be a Borel measurable function. 

The following is Proposition 1 of \citet[p.~57]{mij:etal:2012} for an interior point. 
\begin{proposition} ~ \label{prop:otf}
\begin{enumerate}
\item If $f/\sigma^2$ is locally integrable on $(c,d)$ (i.e.\ integrable on every compact subset of $(c,d)$) then for each $t \in [0,\infty)$
\[
\int_0^t f(X_s) \dd s < \infty \qquad \bP{\gq}\text{-a.s.\ on }\{t < T_{c,d}\}.
\]
\item If $f/\sigma^2$ is not locally integrable on any $(z-\varepsilon,z+\varepsilon)$, $\varepsilon > 0$, then
\[
\int_0^t f(X_s) \dd s = \infty \qquad \bP{\gq}\text{-a.s.\ on }\{T_z < t < T_{c,d}\}.
\]
\end{enumerate} 
\end{proposition}
\begin{proof}
1. Note that, denoting $(L^x_t(X): x\in I, t \in [0,\infty))$ the semimartingale local time process associated with $X$, 
\begin{align}
\label{eq:otf}
\int_0^{t} f(X_u) \dd u &= \int_0^{t} \frac{f(X_u)}{\gs^2(X_u)} \dd \langle X\rangle_u = \int_I \frac{f(x)}{\gs^2(x)} L^x_{t}(X) \dd x, \qquad t \in [0,T_{c,d}),
\end{align}
which follows from the occupation times formula \citep[Ch.~VI, Cor.~1.6, p.~224]{rev:yor:1999}. The conclusion follows from the assumed local integrability of $f/\sigma^2$ and the fact that, a.s.\ on the event $\{t < T_{c,d}\}$, $x \mapsto L_t^x(X)$ is c\`adl\`ag and has compact support.

2. For any $t \in [0,\infty)$, on the event $\{T_z < t < T_{c,d}\}$ it follows from \citet[Thm.~2.7, p.~30]{che:eng:2005} that $L^z_t(X) > 0$ and $L^{z-}_t(x) > 0$ a.s., and hence that there exists $\varepsilon > 0$ such that $x \mapsto L^x_t(X)$ is bounded away from 0 on $(z-\varepsilon,z+\varepsilon)$. On this interval $\int_{z-\varepsilon}^{z+\varepsilon} \frac{f(x)}{\sigma^2(x)} \dd x = \infty$ by assumption, so by \eqref{eq:otf} again we must have $\int_0^t f(X_s) \dd s = \infty$.
\end{proof}
We would like to know whether a point $z$ is separating for $X$; that is, whether the measures $\bP{\gq_0}$, $\bP{\gq_1}$ must separate if $X$ visits the point $z$. From equation \eqref{eq:explode}, the function whose integrability is of relevance is $f=b^2$ as in \eqref{eq:b-int}, and the lesson from \propref{prop:otf} is that we must investigate the local integrability of $f/\sigma^2 = b^2/\sigma^2 = (\mu_{\gq_1} - \mu_{\gq_0})^2/\sigma^4$. Definition \ref{def:non-separating} follows.

\bibliographystyle{myplainnat}
\bibliography{bibliography}

@incollection{pap:rob:2012,
  author =        {O. Papaspiliopoulos and G. O. Roberts},
  booktitle =     {Statistical methods for stochastic differential
                   equations},
  pages =         {311--340},
  publisher =     {Chapman \& Hall},
  series =        {Monographs on Statistics and Applied Probability},
  title =         {Importance sampling techniques for estimation of
                   diffusion models},
  volume =        {124},
  year =          {2012},
}

@article{pap:etal:2013,
  author =        {O. Papaspiliopoulos and G. O. Roberts and O. Stramer},
  journal =       {Journal of Computational and Graphical Statistics},
  number =        {3},
  pages =         {665--688},
  title =         {Data augmentation for diffusions},
  volume =        {22},
  year =          {2013},
}

@article{van:sch:2018,
  author =        {F. {van der Meulen} and M. Schauer},
  journal =       {Stochastics},
  number =        {5},
  pages =         {641--662},
  publisher =     {Taylor \& Francis},
  title =         {Bayesian estimation of incompletely observed
                   diffusions},
  volume =        {90},
  year =          {2018},
}

@article{jen:etal:2023,
  author =        {P. A. Jenkins and M. Pollock and G. O. Roberts},
  journal =       {Methodology and Computing in Applied Probability},
  number =        {4},
  pages =         {83},
  title =         {Flexible {Bayesian} inference for diffusion processes
                   using splines},
  volume =        {25},
  year =          {2023},
}

@book{eth:2011,
  author =        {A. Etheridge},
  publisher =     {Springer Science \& Business Media},
  title =         {Some Mathematical Models from Population Genetics:
                   {\'E}cole D'{\'E}t{\'e} de Probabilit{\'e}s de
                   Saint-Flour XXXIX-2009},
  volume =        {2012},
  year =          {2011},
}

@article{deh:etal:2020,
  author =        {M. Dehasque and M. C. \'{A}vila-Arcos and
                   D. {D\'{\i}ez-del-Molino} and M. Fumagalli and
                   K. Guschanski and E. D. Lorenzen and A. Malaspinas and
                   T. Marques-Bonet and M. D. Martin and G. G. R. Murray and
                   A. S. T. Papadopulos and N. O. Therkildsen and
                   D. Wegmann and L. Dal\'en and A. D. Foote},
  journal =       {Evolution Letters},
  number =        {2},
  pages =         {94--108},
  title =         {Inference of natural selection from ancient {DNA}},
  volume =        {4},
  year =          {2020},
}

@article{dan:etal:2012:PCS,
  author =        {C. E. Dangerfield and D. Kay and K. Burrage},
  journal =       {Procedia Computer Science},
  number =        {1},
  pages =         {1587--1596},
  title =         {Stochastic models and simulation of ion channel
                   dynamics},
  volume =        {1},
  year =          {2012},
}

@article{don:etal:2018,
  author =        {G. {D'Onofrio} and M. Tamborrino and P. Lansky},
  journal =       {Chaos: An Interdisciplinary Journal of Nonlinear
                   Science},
  number =        {10},
  pages =         {103119},
  title =         {The {Jacobi} diffusion process as a neuronal model},
  volume =        {28},
  year =          {2018},
}

@article{del:shi:2002,
  author =        {F. Delbaen and H. Shirakawa},
  journal =       {Asia-Pacific Financial Markets},
  pages =         {191--209},
  title =         {An interest rate model with upper and lower bounds},
  volume =        {9},
  year =          {2002},
}

@article{gou:jas:2006,
  author =        {C. Gourieroux and J. Jasiak},
  journal =       {Journal of Econometrics},
  number =        {1--2},
  pages =         {475--505},
  title =         {Multivariate {Jacobi} process with application to
                   smooth transitions},
  volume =        {131},
  year =          {2006},
}

@article{pal:2011,
  author =        {S. Pal},
  journal =       {Annals of Applied Probability},
  number =        {3},
  pages =         {1180--1213},
  title =         {Analysis of market weights under
                   volatility-stabilized market models},
  volume =        {21},
  year =          {2011},
}

@article{wal:etal:2007,
  author =        {S. G. Walker and S. J. Hatjispyros and T. Nicoleris},
  journal =       {Annals of Applied Probability},
  number =        {1},
  pages =         {67--80},
  title =         {A {Fleming--Viot} process and {Bayesian}
                   nonparametrics},
  volume =        {17},
  year =          {2007},
}

@article{fav:etal:2009,
  author =        {S. Favaro and M. Ruggiero and S. G. Walker},
  journal =       {Electronic Journal of Statistics},
  pages =         {1556--1566},
  title =         {{On a Gibbs sampler based random process in Bayesian
                   nonparametrics}},
  volume =        {3},
  year =          {2009},
}

@article{per:etal:2017,
  author =        {V. Perrone and P. A. Jenkins and D. Span\`o and
                   Y. W. Teh},
  journal =       {Journal of Machine Learning Research},
  number =        {127},
  pages =         {1--45},
  title =         {Poisson random fields for dynamic feature models},
  volume =        {18},
  year =          {2017},
}

@article{sat:1976:JMSJ,
  author =        {K. Sato},
  journal =       {Journal of the Mathematical Society of Japan},
  number =        {4},
  pages =         {621--637},
  title =         {A class of {Markov} chains related to selection in
                   population genetics},
  volume =        {28},
  year =          {1976},
}

@article{daw:1978,
  author =        {D. A. Dawson},
  journal =       {The Canadian Journal of Statistics / La Revue
                   Canadienne de Statistique},
  number =        {2},
  pages =         {143--168},
  title =         {Geostochastic calculus},
  volume =        {6},
  year =          {1978},
}

@book{eth:2000,
  author =        {A. M. Etheridge},
  publisher =     {American Mathematical Society},
  series =        {University Lecture Series},
  title =         {An introduction to superprocesses},
  volume =        {20},
  year =          {2000},
}

@incollection{che:uru:2006,
  author =        {A. Cherny and M. Urusov},
  booktitle =     {From stochastic calculus to mathematical finance},
  pages =         {125--168},
  publisher =     {Springer},
  title =         {On the absolute continuity and singularity of
                   measures on filtered spaces: separating times},
  year =          {2006},
}

@article{mij:uru:2012:PTRF,
  author =        {A. Mijatovi{\'c} and M. Urusov},
  journal =       {Probability Theory and Related Fields},
  number =        {1},
  pages =         {1--30},
  title =         {On the martingale property of certain local
                   martingales},
  volume =        {152},
  year =          {2012},
}

@article{cri:uru:2026,
  title={Separating times for one-dimensional general diffusions},
  author={D. Criens and M. Urusov},
  journal={Annals of Applied Probability},
  volume={36},
  number={1},
  pages={1--45},
  year={2026}
}

@book{rev:yor:1999,
  author =        {D. Revuz and M. Yor},
  note =          {Third edition},
  publisher =     {Springer},
  title =         {Continuous martingales and {Brownian} motion},
  year =          {1999},
}

@article{shi:wat:1973,
  author =        {T. Shiga and S. Watanabe},
  journal =       {Zeitschrift f{\"u}r Wahrscheinlichkeitstheorie und
                   verwandte Gebiete},
  number =        {1},
  pages =         {37--46},
  title =         {Bessel diffusions as a one-parameter family of
                   diffusion processes},
  volume =        {27},
  year =          {1973},
}

@incollection{pit:yor:1981,
  author =        {J. Pitman and M. Yor},
  booktitle =     {Stochastic Integrals},
  editor =        {D. Williams},
  pages =         {285--370},
  publisher =     {Springer-Verlag},
  series =        {Lecture Notes in Mathematics},
  title =         {Bessel processes and infinitely divisible laws},
  volume =        {851},
  year =          {1981},
}

@article{xue:1990,
  author =        {X. Xue},
  journal =       {S{\'e}minaire de probabilit{\'e}s de Strasbourg},
  pages =         {137--153},
  title =         {A zero-one law for integral functionals of the
                   {Bessel} process},
  volume =        {24},
  year =          {1990},
}

@article{ove:1998,
  author =        {L. Overbeck},
  journal =       {Scandinavian Journal of Statistics},
  number =        {1},
  pages =         {111--126},
  title =         {Estimation for continuous branching processes},
  volume =        {25},
  year =          {1998},
}

@article{che:2001,
  author =        {A. S. Cherny},
  journal =       {Theory of Probability \& Its Applications},
  number =        {2},
  pages =         {195--209},
  title =         {Convergence of some integrals associated with
                   {Bessel} processes},
  volume =        {45},
  year =          {2001},
}

@article{wat:1979,
  author =        {G. A. Watterson},
  journal =       {Advances in Applied Probability},
  number =        {1},
  pages =         {14--30},
  title =         {Estimating and testing selection: the two-alleles,
                   genic selection diffusion model},
  volume =        {11},
  year =          {1979},
}

@article{san:etal:2022,
  author =        {J. Sant and P. A. Jenkins and J. Koskela and
                   D. Span\`o},
  journal =       {Scandinavian Journal of Statistics},
  number =        {4},
  pages =         {1728--1760},
  title =         {Convergence of likelihood ratios and estimators for
                   selection in nonneutral {Wright--Fisher} diffusions},
  volume =        {49},
  year =          {2022},
}

@book{bre:1968,
  author =        {L. Breiman},
  publisher =     {Addision--Wesley},
  title =         {Probability},
  year =          {1968},
}

@article{hob:rog:1998,
  author =        {D. G. Hobson and L. C. G. Rogers},
  journal =       {Mathematical Finance},
  number =        {1},
  pages =         {27--48},
  title =         {Complete models with stochastic volatility},
  volume =        {8},
  year =          {1998},
}

@article{ruf:2015,
  author =        {J. Ruf},
  journal =       {Electronic Communications in Probability},
  pages =         {1--10},
  title =         {The martingale property in the context of stochastic
                   differential equations},
  volume =        {20},
  year =          {2015},
}

@article{cri:gla:2018,
  author =        {D. Criens and K. Glau},
  journal =       {Electronic Journal of Probability},
  pages =         {1--28},
  title =         {Absolute continuity of semimartingales},
  volume =        {23},
  year =          {2018},
}

@article{des:etal:2021,
  author =        {S. Desmettre and G. Leobacher and L. C. G. Rogers},
  journal =       {Finance and Stochastics},
  pages =         {359--381},
  title =         {Change of drift in one-dimensional diffusions},
  volume =        {25},
  year =          {2021},
}

@article{eth:kur:1993,
  author =        {S. N. Ethier and T. G. Kurtz},
  journal =       {SIAM Journal of Control and Optimization},
  number =        {2},
  pages =         {345--386},
  title =         {{Fleming--Viot} processes in population genetics},
  volume =        {31},
  year =          {1993},
}

@article{cri:2021,
  author =        {D. Criens},
  journal =       {Electronic Journal of Probability},
  pages =         {1--26},
  title =         {On absolute continuity and singularity of
                   multidimensional diffusions},
  volume =        {26},
  year =          {2021},
}

@article{pit:yor:1982,
  author =        {J. Pitman and M. Yor},
  journal =       {Zeitschrift f\"ur Wahrscheinlichkeitstheorie und
                   Verwandte Gebiete},
  pages =         {425--457},
  title =         {A decomposition of {Bessel} bridges},
  volume =        {59},
  year =          {1982},
}

@article{sch:etal:2016,
  author =        {J. G. Schraiber and S. N. Evans and M. Slatkin},
  journal =       {Genetics},
  pages =         {493--511},
  title =         {Bayesian inference of natural selection from allele
                   frequency time series},
  volume =        {203},
  year =          {2016},
}

@article{gri:etal:2018,
  author =        {R. C. Griffiths and P. A. Jenkins and D. Span\`o},
  journal =       {Theoretical Population Biology},
  pages =         {67--77},
  title =         {{Wright--Fisher} diffusion bridges},
  volume =        {122},
  year =          {2018},
}

@book{rog:wil:2000:II,
  author =        {L. C. G. Rogers and D. Williams},
  publisher =     {Cambridge university press},
  title =         {Diffusions, {Markov} processes and martingales:
                   {Volume 2, It{\^o} calculus}},
  volume =        {2},
  year =          {2000},
}

@article{daw:1968,
  author =        {D. A. Dawson},
  journal =       {Transactions of the American Mathematical Society},
  number =        {1},
  pages =         {1--31},
  title =         {Equivalence of {Markov} processes},
  volume =        {131},
  year =          {1968},
}

@article{ike:wat:1977,
  author =        {N. Ikeda and S. Watanabe},
  journal =       {Osaka Journal of Mathematics},
  number =        {3},
  pages =         {619--633},
  title =         {A comparison theorem for solutions of stochastic
                   differential equations and its applications},
  volume =        {14},
  year =          {1977},
}

@article{che:str:2010,
  author =        {L. Chen and D. W. Stroock},
  journal =       {SIAM Journal on Mathematical Analysis},
  number =        {2},
  pages =         {539--567},
  title =         {The fundamental solution to the {Wright--Fisher}
                   equation},
  volume =        {42},
  year =          {2010},
}

@article{eps:maz:2010,
  author =        {C. L. Epstein and R. Mazzeo},
  journal =       {SIAM Journal on Mathematical Analysis},
  number =        {2},
  pages =         {568--608},
  title =         {{Wright--Fisher} diffusion in one dimension},
  volume =        {42},
  year =          {2010},
}

@book{bas:pra:1980,
  author =        {I. V. Basawa and B. L. S. {Prakasa Rao}},
  publisher =     {Academic Press},
  series =        {Probability and mathematical statistics},
  title =         {Statistical inference for stochastic processes},
  year =          {1980},
}

@book{klo:etal:2003,
  author =        {P. E. Kloeden and E. Platen and H. Schurz},
  note =          {Third printing},
  publisher =     {Springer},
  title =         {Numerical solution of {SDE} through computer
                   experiments},
  year =          {2003},
}

@article{gri:jen:2023,
  author =        {R. C. Griffiths and P. A. Jenkins},
  journal =       {Journal of Mathematical Biology},
  pages =         {98},
  title =         {An estimator for the recombination rate from a
                   continuously observed diffusion of haplotype
                   frequencies},
  volume =        {86},
  year =          {2023},
}

@article{mij:uru:2015,
  author =        {A. Mijatovi{\'c} and M. Urusov},
  journal =       {Journal of Theoretical Probability},
  number =        {3},
  pages =         {892--922},
  title =         {On the loss of the semimartingale property at the
                   hitting time of a level},
  volume =        {28},
  year =          {2015},
}

@article{wei:win:1990,
  author =        {C. Z. Wei and J. Winnicki},
  journal =       {Annals of Statistics},
  number =        {4},
  pages =         {1757--1773},
  title =         {Estimation of the means in the branching process with
                   immigration},
  volume =        {18},
  year =          {1990},
}

@book{kut:2004,
  author =        {Y. A. Kutoyants},
  publisher =     {Springer-Verlag London},
  title =         {Statistical inference for ergodic diffusion
                   processes},
  year =          {2004},
}

@article{ben:keb:2012,
  author =        {M. {Ben Alaya} and A. Kebaier},
  journal =       {Stochastic Models},
  number =        {4},
  pages =         {609--634},
  title =         {Parameter estimation for the square-root diffusions:
                   ergodic and nonergodic cases},
  volume =        {28},
  year =          {2012},
}

@article{ben:keb:2013,
  author =        {M. {Ben Alaya} and A. Kebaier},
  journal =       {Stochastic Analysis and Applications},
  number =        {4},
  pages =         {552--573},
  title =         {Asymptotic behavior of the maximum likelihood
                   estimator for ergodic and nonergodic square-root
                   diffusions},
  volume =        {31},
  year =          {2013},
}

@incollection{mij:etal:2012,
  title={Martingale property of generalized stochastic exponentials},
  author={A. Mijatovi{\'c} and N. Novak and M. Urusov},
  booktitle={S{\'e}minaire de Probabilit{\'e}s XLIV},
  pages={41--59},
  year={2012},
  publisher={Springer}
}

@Book{kar:tay:1981,
	author =	{S. Karlin and H. M. Taylor},
	title =	{A second course in stochastic processes},
	year =	1981,
	publisher =	{Academic Press},
	address =	{New York}
}

@book{kal:2021,
	title =	{Foundations of modern probability},
	author =	{O. Kallenberg},
	year =	2021,
	edition =	{3rd},
	publisher =	{Springer}
}

@Article{asc:etal:2023,
	author =	{F. Ascolani and A. Lijoi and M. Ruggiero},
	title =	{Smoothing distributions for conditional {Fleming--Viot} and {Dawson--Watanabe} diffusions},
	journal =	{Bernoulli},
	year =	2023,
	volume =	29,
	number =	2,
	pages =	{1410--1434}
}

@article{asc:etal:2024,
author = {F. Ascolani and S. Damato and M. Ruggiero},
title = {An {R} package for nonparametric inference on dynamic populations with infinitely many types},
journal = {Journal of Computational Biology},
volume = {31},
number = {12},
pages = {1305--1311},
year = {2024}
}

@Article{rob:str:2001,
	author =	{G. O. Roberts and O. Stramer},
	title =	{On inference for partially observed nonlinear diffusion models using the {Metropolis-Hastings} algorithm},
	journal =	{Biometrika},
	year =	2001,
	volume =	88,
	number =	3,
	pages =	{603--621}
}

@Book{kar:shr:1998,
	author =	{I. Karatzas and S. E. Shreve},
	title =	{Brownian motion and stochastic calculus},
	year =	1998,
	publisher =	{Springer},
	edition =	{2nd}
}

@book{che:eng:2005,
  title={Singular stochastic differential equations},
  author={A. S. Cherny and H. Engelbert},
  series = {Lecture notes in mathematics},
  number={1858},
  year={2005},
  publisher={Springer}
}

@book{jac:shi:2003,
  title={Limit theorems for stochastic processes},
  author={J. Jacod and A. Shiryaev},
  volume={288},
  year={2003},
edition = {2nd},
  publisher={Springer Science \& Business Media}
}

@article{cri:uru:2025,
  title={On the representation property for 1D general diffusion semimartingales},
  author={D. Criens and M. Urusov},
  journal={Theory of Probability \& Its Applications},
  volume={69},
  number={4},
  pages={579--591},
  year={2025}
}
\end{document}